\documentclass[]{article}
\usepackage{import}
\makeatletter
\def\thm@space@setup{%
\thm@preskip=1em \thm@postskip=0pt
}
\makeatother

\usepackage[affil-it]{authblk}

\usepackage{amsmath, amssymb, bbold, mathtools}
\usepackage[ruled, vlined]{algorithm2e}
\usepackage{multirow}
\usepackage{graphicx}
\usepackage{tcolorbox}
\usepackage[]{hyperref}

\usepackage[shortlabels]{enumitem}

\usepackage[square,numbers]{natbib}
\usepackage{tikz}
\usetikzlibrary{calc}
\usetikzlibrary{shapes}

\usepackage{amsthm}
\theoremstyle{plain}
\newtheorem{theorem}{Theorem}[section]
\newtheorem{lemma}[theorem]{Lemma}

\theoremstyle{definition}
\newtheorem{definition}[theorem]{Definition}
\newtheorem{remark}[theorem]{Remark}
\newtheorem{example}[theorem]{Example}
\newtheorem{proposition}[theorem]{Proposition}
\newtheorem{assumption}[theorem]{Assumption}

\usepackage{acronym}
\acrodef{wlog}[WLOG]{without loss of generality}
\acrodef{lsc}[lsc]{lower semi-continuous}

\definecolor{red}{RGB}{163, 31, 52}
\definecolor{gray}{RGB}{194, 192, 191}
\definecolor{blue}{RGB}{59, 89, 152}
\definecolor{green}{RGB}{0, 179, 0}

\newcommand{\norm}[1]{\left\|#1\right\|}

\newcommand{\iprod}[2]{\left<#1, #2\right>}
\newcommand{\abs}[1]{\left|#1\right|}
\newcommand{\E}[2]{\mathbf{E}_{#1} \left[ #2 \right]}
\newcommand{\one}[1]{\mathbb{1}\left\{#1\right\}}
\renewcommand{\tfrac}[2]{{#1}/{#2}}

\newcommand{\set}[2]{\left\{ #1\ : \ #2 \right\}}
\newcommand{\tset}[2]{\{ #1\ : \ #2 \}}

\newcommand{\defn}[0]{:=}
\newcommand{\KL}[0]{{\rm{KL}}}
\newcommand{\LP}[2]{{\rm{LP}}(#1,#2)}
\newcommand{\ml}[0]{{\rm{ml}}}

\newcommand{\Prob}[0]{{\rm{Prob}}}

\newcommand{\mc}{\mathcal}

\newcommand{\mr}{\mathrm}

\renewcommand{\d}{{\mathrm{d}}}

\renewcommand{\Re}{\mathrm{R}}
\newcommand{\nRe}{\Re\cup\{+\infty\}}

\newcommand{\st}{\mr{s.t.}}

\renewcommand{\emph}{\textbf}

\DeclareMathOperator{\cl}{cl}
\DeclareMathOperator{\interior}{int}

\usepackage[margin=2cm]{geometry}
\usepackage{tikz}
\usepackage{pgfplots}
\usetikzlibrary{external}
\usetikzlibrary{calc}
\usetikzlibrary{shapes.multipart}
\usetikzlibrary{decorations.pathreplacing}
\usetikzlibrary{arrows.meta}
\usepgfplotslibrary{fillbetween}
\usetikzlibrary{patterns}

\usepackage{setspace}
\spacing{1.5}

\title{Efficient Data-Driven Optimization with Noisy Data}
\author[1]{Bart P.G.\ \mbox{Van Parys}\thanks{\href{mailto:vanparys@mit.edu}{vanparys@mit.edu}}}
\affil[1]{Sloan School of Management, MIT}

\begin{document}

\maketitle

\begin{abstract}
  Classical Kullback-Leibler or entropic distances have recently been shown \cite{vanparys2020data} to enjoy certain desirable statistical properties in the context of decision-making with noiseless data. {\color{black} However, in most practical situations the data available to a decision maker is subject to a certain amount of measurement noise}. We hence study here data-driven prescription problems in which the data is corrupted by a {\color{black} known noise source}. We derive efficient data-driven formulations in this noisy regime and indicate that they enjoy an entropic optimal transport interpretation. Finally, we show that these efficient robust formulations are tractable in several interesting settings by exploiting a classical representation result by \citet{strassen1965existence}.
\end{abstract}

\section{Introduction}
\label{sec:introduction}

{\color{black}
Let $\mc P$ be a family of probability measures over a space $\Xi$ and let $P$ be an unknown probability measure in this family.
We are interested in finding an $0<\epsilon$-suboptimal%
\footnote{Here $\arg_\epsilon \inf_{z\in Z}\, \E{P}{\ell(z, \xi)}$ with $\epsilon>0$ corresponds to the set $\set{z\in Z}{\E{P}{\ell(z, \xi)} <  \inf_{z\in Z} \E{P}{\ell(z, \xi)}+ \epsilon }$.}
solution to the stochastic optimization problem
\begin{equation}
  \label{eq:sop}
  z(P)\in \arg_\epsilon{\inf_{z\in Z}} ~\{\E{P}{\ell(z, \xi)} =\textstyle\int_\Xi \ell(z, \xi) \, \d P(\xi)\}.
\end{equation}
A wide spectrum of decision problems can be cast as instances of problem ~\eqref{eq:sop}. Shapiro et al.\ \cite{shapiro2014lectures} point out, that~\eqref{eq:sop} can be viewed as the first stage of a two-stage stochastic program, where the loss function $\ell:Z\times \Xi\to\Re$ embodies the optimal value of a subordinate second-stage problem. Alternatively, problem~\eqref{eq:sop} may also be interpreted as a generic learning problem in the spirit of statistical learning theory.

In practice, however, the distribution $P$ which describes the uncertain parameter $\xi$ is not known. The stochastic optimization community has in recent years focused on making decisions directly based on  a finite collection of independent data observations
\begin{equation}
  \label{eq:noiseless-observations}
  \xi_i \sim P \quad \forall i \in [1, \dots, N],
\end{equation}
instead. In this paper we will denote this observational model, considered by the overwhelming majority of the literature on data-driven decisions, described in Equation (\ref{eq:noiseless-observations}) as \textit{noiseless}. That is, the decision maker has access to uncorrupted independent samples from the probability measure of interest $P$.
However, it is often the case in practice that the data itself is not observed directly but rather through a measurement device with known limitations. Such noisy data is better known as censored data in statistics \cite{gijbels2010censored}.
An observational model in which the noisy data
\begin{equation}
  \label{eq:noisy-observations}
  \xi'_i = \xi_i+n_i \quad \forall i \in [1, \dots, N],
\end{equation}
is observed instead where the noise terms $n_i$ are independent and share a known distribution has been widely studied in the system control and identification literature.
We will come back to this observational model and further generalizations in Section \ref{sec:noisy-data}.
In the remainder of this section we first briefly discuss various results of interest in the noiseless observational model.

In case the probability measure $P$ is only known to belong to the probability simplex $\mc P(\Xi)$, one reasonable substitute for $P$ could be its empirically observed counterpart denoted here as $P_N\defn\sum_{i=1}^N \delta_{\xi_i}/N$.
If on the other hand some prior information is available in the sense that the probability measure $P$ is known to belong to a subset $\mc P\subset \mc P(\Xi)$, a maximum likelihood estimate \cite{groeneboom1992information} may be considered  instead.
In the machine learning and robust optimization community such point estimates are widely known to be problematic when used naively in subsequent analysis.
In particular, it is widely established both empirically as well as in theory that a sample average formulation
\begin{equation}
  \label{eq:saa}
  z(P_N)\in \arg_\epsilon{\inf_{z\in Z} \,\E{P_N}{\ell(z, \xi)}}
\end{equation}
which substitutes $P$ with a mere point estimate $P_N$ tends to disappoint out of sample. That is, the actual cost ($\E{P}{\ell(z(P_N), \xi)}$) observed out of sample exceeds the predicted cost $(\E{P_N}{\ell(z(P_N), \xi)})$ of the data-driven decision $z(P_N)$. This adversarial phenomenon is well known  colloquially as the ``Optimizer's Curse'' \cite{michaud1989optimizationengine} and is akin to the overfitting phenomenon in the context of prediction problems.
Such adversarial phenomena related to over-calibration to observed data but poor performance on out-of-sample data can be attributed primarily to the treatment of mere point estimates as exact substitutes for the unknown probability measure.

Ambiguity sets consisting of all probability measures sufficiently compatible with the observed data can offer a better alternative to simple point estimates.
As the data observations are here independent and identically distributed, their order is irrelevant, and ambiguity sets $\mc A_N(P_N) \subseteq \mc P$ can be made functions of the empirical probability measure $P_N$ rather than the data itself.
A large line of work in the robust optimization community, pioneered by \citet{scarf1958min}, focuses consequently on data-driven formulations of the form
\begin{equation}
  \label{eq:gen-dro}
  z_{\mc A}(P_N)\in \arg_\epsilon\inf_{z\in Z}\, \sup \,\set{\E{P}{\ell(z, \xi)}}{P\in \mc A_N(P_N)}
\end{equation}
which can be thought of as robust counterparts to the nominal sample average formulation stated in Equation \eqref{eq:saa} and where by convention consider the supremum to take the value $-\infty$ in case its feasible set $\mc A_N(P_N)$ is empty.
The recent uptick in popularity of such robust formulations is in no small part due to the fact that they are often just as tractable and typically enjoy superior statistical properties than their nominal counterpart.
Earlier work \citep{bertsimas2005optimal, delage2010distributionally, goh2010distributionally, wiesemann2014distributionally, vanparys2016generalized} focused on ambiguity sets consisting of probability measures sharing certain given moments.
More recent approaches \citep{bayraksan2015data, bertsimas2018data, hu2013kullback, gotoh2018robust, gupta2019near} however consider ambiguity sets $\mc A_N(P _N)=\set{P\in \mc P}{D(P_N, P)\leq r_N}$ which are based on some statistical distance $D:\mc P(\Xi)\times \mc P(\Xi)\to\nRe$\footnote{In this paper we will use \textit{distance} in a rather informal manner as indeed the distances we will encounter are neither symmetric nor satisfy a triangular inequality.}.
Such ambiguity sets are interpretable as the set of probability measures sufficiently close to the empirical probability measure $P_N$.
Two qualitatively different distances have recently positioned themselves as the front runners for data-driven decision-making and are now briefly discussed.

The optimal transport distance between a measure $\mu$ on $\Xi$ and a measure $\nu$ on another set $\Xi'$ can be defined as
\begin{equation}
  \label{def:wasserstein}
  W_0(\mu, \nu) \defn \inf_{T\in \mc T(\mu, \nu)} \int_{\Xi\times \Xi'} d(\xi , \xi')\, \d T(\xi, \xi')
\end{equation}
for a given the cost function $d:\Xi\times \Xi'\mapsto \nRe$. Here
\(
  \mc T(\mu, \nu) \defn \set{T\in \mc P(\Xi\times\Xi')\!}{\!\Pi_{\Xi'} T=\mu, \,\Pi_{\Xi}T=\nu}
\)
is the so called transport polytope and consists of all transport measures with given marginal measures $\mu$ and $\nu$. Assume for a moment that $\Xi=\Xi'=\Re^m$ and $d(\xi, \xi')=\norm{\xi-\xi'}_2$. Then, the optimal transport distance coincides with the classical Wasserstein distance between probability measures \citep{santambrogio2015optimal}.
Optimal transport distances  have received a lot of attention both in the context of prescriptive analytics \citep{esfahani2018data,chen2018robust,xie2019tractable,shafieezadeh2019regularization,blanchet2019robust} as well as in the machine learning community at large \citep{kuhn2019wasserstein, peyre2019computational}.
In the context of prescriptive analytics such distances have become very popular after the seminal work \cite{esfahani2018data} pointed out that the resulting robust formulation 
\(
z_{W,r}(P_N)\in \arg_\epsilon\inf_{z\in Z}\, \sup \,\set{\E{P}{\ell(z, \xi)}}{W_0(P_N, P)\leq r}
\)
need not be intractable and enjoys strong out-of-sample guarantees.  For $r<1$ and $2\leq \dim(\Xi)<\infty$, \citet{kuhn2019wasserstein} points out that we have for any $P$ with $A=\E{P}{\exp(\norm{\xi}_2^a)}<\infty$ for some $a>1$ that with $r'=(\tfrac{r}{C})^{\tfrac{1}{\dim(\Xi)}}$ we have
\begin{equation}
  \label{eq:dissapoint-wasserstein}
  \varlimsup_{N\to\infty} \frac 1N \log \Prob \left[\E{P}{\ell(z_{W,r'}(P_N), \xi)} > c_{W,r'}(z_{W,r'}(P_{N}), P_{N}) \right] \leq- r
\end{equation}
where $C$ is a known positive constant depending only on $a$, $A$ and $\dim(\Xi)$. In other words, the probability of being disappointed (the actual cost $\E{P}{\ell(z_{W,r'}(P_N), \xi)}$ of the decision $z_{W,r'}(P_N)$ exceeds the predicted worst-case cost $c_{W,r'}(z_{W,r'}(P_{N}), P_{N})$) decays exponentially fast in the number of samples.

Optimal transport distances are however computationally quite challenging as even determining the optimal transport distance between two given measures requires the solution of the associated linear optimization problem in Equation (\ref{def:wasserstein}).
In fact, optimal transport distances are often computed numerically, c.f., \citep{cuturi2013sinkhorn}, by considering their embedding in the larger family of entropic optimal transport distances
\begin{equation}
  \label{def:entropic-wasserstein}
  W_\epsilon(\mu, \nu) \defn \inf_{T\in \mc T(\mu, \nu)} \int_{\Xi\times \Xi'} d(\xi , \xi') \d T(\xi, \xi') + \epsilon \KL(T, \Pi_\Xi T\otimes \Pi_{\Xi'}T)
\end{equation}
as the limit $\epsilon \downarrow 0$ where we denote the entropic divergence between two finite measures $\mu$ and $\nu$ on the same space as
\[
  \KL(\mu, \nu)=
  \begin{cases}
    \int \log\left(\frac{\d \mu}{\d \nu}\right) \d \mu - \int \d \mu+\int \d\nu & {\rm{if}}~ \mu \ll \nu,\\[0.25em]
    +\infty & {\rm{otherwise}}
  \end{cases}
\]
where the random variable $\tfrac{\d \mu}{\d \nu}$ denotes the Radon-Nikodym derivative between $\mu$ and $\nu$.
The previously discussed entropic divergence is a particular member of the class of convex $f$-divergences which like the Wasserstein distance are well known \cite{hu2013kullback} to yield tractable robust formulations.
Moreover, when the loss function $\ell(z, \xi)$ is continuous on the compact set $Z\times\Xi$ the associated robust prescriptive formulation \eqref{eq:gen-dro} can be specialized to 
\(
z_{\KL,r}(P_N)\in \min_{z\in Z} \max \,\set{\E{P}{\ell(z, \xi)}}{\KL(P_N, P)\leq r}
\)
and enjoys strong statistical out-of-sample guarantees. One can prove \citep[Theorem 11]{vanparys2020data} for all $P\in \mc P(\Xi)$ that indeed
\begin{equation}
  \label{eq:dissapoint-kl}
  \varlimsup_{N\to\infty} \frac 1N \log \Prob \left[\E{P}{\ell(z_{\KL,r}(P_N), \xi)} >  \sup \,\set{\E{Q}{\ell(z_{\KL,r}(P_N), \xi)}+\epsilon}{\KL(P_N, Q)\leq r} \right] \leq -r.
\end{equation}
for any $\epsilon>0$.
In other words, the probability of being disappointed (the actual cost $\E{P}{\ell(z_{\KL,r}(P_N), \xi)}$ of the decision $z_{\KL,r}(P_N)$ exceeds the predicted worst-case cost $\sup \set{\E{Q}{\ell(z_{\KL,r}(P_N), \xi)}}{\KL(P_N, Q)\leq r}$ by any small amount $\epsilon>0$) decays exponentially fast in the number of samples with rate precisely the size $r$ of the considered ambiguity set.

Equations (\ref{eq:dissapoint-wasserstein}) and (\ref{eq:dissapoint-kl}) reflect the fact that when properly calibrated robust entropic and Wasserstein formulations enjoy essentially the same out-of-sample guarantees. The classical sample average formulation can also be made to enjoy similar out-of-sample guarantees by naively inflating its objective in Equation (\ref{eq:saa}) by some bias term $b>0$. Indeed, taking the bias term $b$ sufficiently large we can guarantee that
  \begin{equation}
    \label{eq:dissapoint-saa}
    \varlimsup_{N\to\infty} \frac 1N \log \Prob \left[\E{P}{\ell(z_N(P_N), \xi)} > \E{P_N}{\ell(z_N(P_N), \xi)} + b \right]   \leq -r ~~ \forall P\in \mc P.
  \end{equation}
  When several prescription formulations enjoy the same out-of-sample guarantees, we should prefer that formulation which inflates the cost prediction the least.
  \citet[Theorem 11]{vanparys2020data} shows that any formulation which enjoys the out-of-sample guarantee
  \begin{equation}
    \label{eq:dissapoint-gen}
    \varlimsup_{N\to\infty} \frac 1N \log \Prob \left[\E{P}{\ell(z_N(P_N), \xi)} > \tilde c_r(\tilde z_r(P_N), P_N)  \right]   \leq -r ~~ \forall P\in \mc P.
  \end{equation}
  must be more conservative in its cost predictions compared to an entropic formulation, i.e.,
  \[
    \tilde c_r(\tilde z_r(P_N), P_N)  \geq \sup \,\set{\E{Q}{\ell(z_{\KL,r}(P_N), \xi)}}{{\KL}(P_N, Q)\leq r} \geq \E{P_N}{\ell(z(P_N), \xi)} 
  \]
  indicating that the entropic formulation is universally least conservative or \textit{efficient}.
  We remark however that the efficiency of the entropic formulation is intimately tied to the noiseless data model. Indeed, \citet{vanparys2020data} establish the efficiency of the entropic formulation by pointing out that the Kullback-Leibler divergence is precisely the rate function characterizing large deviations between the empirical distribution $P_N$ and the noiseless data generating distribution $P$ \cite{dembo2009large}. 
 }

\subsection{Notation}

We will assume that $\Xi$ and $\Xi'$ are Polish topological spaces and hence so is the product space $\Xi\times \Xi'$ when equipped with the product topology. Given any set $S\subset \Xi'$ we denote with $S^\delta=\set{s'\in \Xi'}{\norm{s-s'}\leq \delta, ~s\in S}$ its $\delta$-inflation.
We denote with $\mc M_+(\Xi)$, $\mc M_+(\Xi')$ and $\mc M_+(\Xi\times\Xi')$ the sets of all positive Borel measures on the spaces $\Xi$, $\Xi'$ and $\Xi\times\Xi'$, respectively. Similarly, we denote with $\mc P(\Xi)$, $\mc P(\Xi')$ and $\mc P(\Xi\times\Xi')$ the sets of all Borel probability measures on the spaces $\Xi$, $\Xi'$ and $\Xi\times\Xi'$, respectively. Given two measures $\mu$ and $\nu$ we denote with $\mu\otimes \nu$ as their product measure.
Following \citet[Section 6.2]{dembo2009large} the probability simplices $\mc P(\Xi)$, $\mc P(\Xi')$ and $\mc P(\Xi\times\Xi')$ when equipped with the topology of weak convergence of probability measures are Polish spaces too.

\subsection{Contributions}

In this paper we generalize the framework of efficient formulations introduced by \cite{vanparys2020data} to problems with noisy data.
We present three contributions:

\begin{enumerate}
\item We prove in Theorem \ref{cor:optimal-dro} that the family of robust formulations
  \begin{equation}
    \label{eq:dro-smooth}
    \tilde z^{\delta} (P'_N) \in  \arg_\epsilon\inf_{z\in Z}\sup \set{\E{Q}{\ell(z, \xi)}}{Q\in \mc P, ~I^\delta(P'_N, Q)\leq r}
  \end{equation}
  parameterized in $\delta>0$ and where $P_N'=\sum_{i=1}^N \delta_{\xi'_i}/N$ denotes the empirical distribution of the noisy data essentially dominates all other formulations mirroring the notion of efficiency enjoyed by the entropic distance in the noiseless regime. Perhaps surprisingly we show that the rate function $I$ in the noisy setting is too irregular and its 
  $\delta$-smoothed counterparts
  \begin{equation}
    \label{eq:smooth-rate-function}
    I^\delta(\mu, \nu) \defn \inf\set{ I(\upsilon, \nu) }{\upsilon\in \mc P(\Xi'), ~\LP{\upsilon}{\mu}\leq \delta}
  \end{equation}
  must be considered instead. {\color{black} This implies in particular that in stark contrast to the noiseless setting there is not any longer a single most efficient formulation but rather a family of increasingly more efficient formulations.}

\item {\color{black}We show in Section \ref{sec:consistency} that by reducing the smoothing parameter $\delta_N$ and robustness parameter $r_N$ at an appropriate rate, consistent formulations with finite sample guarantees can be obtained under an identifiability condition and a mild assumption on the loss function $\ell$.}

\item We state tractable reformulations of the proposed novel family of efficient robust prescriptive formulations \eqref{eq:dro-smooth} under certain technical conditions on the loss function $\ell$ and space $\Xi$ in Section \ref{sec:tract-pred-form}. In particular, we exploit a classical representation result of \citet{strassen1965existence} which seems to be novel in this context and derive a dual formulation whose size is independent of the event set $\Xi'$. 

\end{enumerate}

\section{Decision-Making with Noisy Data}
\label{ssec:optimal-data-analytics}

\subsection{Noisy Data}
\label{sec:noisy-data}

As stated in Equation (\ref{eq:noisy-observations}) the distribution $P'$ of the noisy data $\xi'$ may be distinct from the distribution $P$ of the unobserved noiseless data $\xi$. We will allow the noisy data to takes value in $\Xi'$ which may be different from $\Xi$. That is, in general the noisy observations may take value in a different set than the noiseless observations.
We only assume here that the noisy observations are drawn independently as
\[
  \xi'_i \sim O_{\xi_i} \quad \forall i \in [1, \dots, N].
\]
{\color{black}That is, each noisy data point $\xi'_i$ is obtained as an independent draw from a distribution $O_{\xi_i}\in \mc P(\Xi')$ given a noiseless observation $\xi_i$}. We stress here again that we assume that the mapping $O:\Xi\to\mc P(\Xi')$ which characterizes our observational model is given. In other words, the distributional nature of the noise corrupting the unobserved data points is known. We will refer to $O$ as our observational model as it precisely characterizes how the noisy data is derived from the noiseless data.

{\color{black}To establish the decomposition result in Theorem \ref{thm:ldp} the observational model will be required to satisfy the following technical condition.}
\begin{assumption}
  \label{ass:density}
  The measure $O_\xi$ is absolutely continuous with respect to a base measure $m'$ for all $\xi\in \Xi$, i.e., $O_\xi\ll m'$ for all $\xi\in \Xi$. Furthermore, there exists a measurable function $d:\Xi\times \Xi'\to \Re$ so that
  \(
    \tfrac{\d O_{\xi}}{\d m'}(\xi') = \exp(-d(\xi, \xi')) 
    \)
 for all  $\xi'\in \Xi'$.
\end{assumption}

The relationship between the probability measure $P'$ of the noisy observations and the probability measure $P$ of the unobserved noiseless data can be characterized as the convolution $P'\defn O \star P$ and is given explicitly as
\(
  P'(B) = (O \star P)(B) \defn \int O_\xi(B) \d P(\xi)
\)
for all measurable sets $B \in \mc B(\Xi')$. We will denote with the set $\mc P'\defn \set{O\star P\in \mc P(\Xi')}{P\in \mc P}$ the family of potential distributions of our noisy data.

We point out that this noisy data model is quite flexible and captures a wide variety of settings. 

{\color{black}

\begin{example}[Additive Noise]
  \label{ex:additive-error}
  Practical measurements are typically corrupted by some amount of measurement error.
  We consider here independent additive measurement error $e_i$ from some distribution $E\in \mc P(\Xi')$ with $\Xi'=\Re^{\dim(\Xi)}$ as an example.
  In this case we observe the noisy data
  \(
    \xi_i' = \xi_i+e_i
  \)
  instead of the data $\xi_i$ itself. This observational model is characterized by the map $O^{AE}: \xi \mapsto E(\xi)$ where $E(\xi)$ denotes the error distribution translated by $\xi$, i.e., $E(\xi)[B]  =\int \one{\xi+e\in B} \d E(e)$ for every measurable set $B$ in $\Xi'$.
\end{example}
  
\begin{example}[Gaussian Noise]
  \label{ex:gaussian-noise}
  Perhaps the most classical example of the additive noise model discussed previously is the simple case of independent zero mean Gaussian additive noise $z_i$ with variance $\sigma^2I$.
  This observational model is characterized by the map $O^{GN}: \xi \mapsto N(\xi, \sigma^2 I)$ where $N(\mu, V)$ denotes here a normal distribution with mean vector $\mu\in \Re^{\dim(\Xi)}$ and variance matrix $V\in \Re^{\dim(\Xi)\times \dim(\Xi)}$.
  Assumption \ref{ass:density} holds for $O^{GN}$ with $\Xi'=\Re^{\dim(\Xi)}$ and $d(\xi, \xi')=\norm{\xi-\xi'}_2^2/(2\sigma^2)$ and $m'=\tfrac{\mu'}{(\sigma \sqrt{(2\pi)^{\dim(\Xi')}})}$ with $\mu'$ the Lebesgue measure on $\Xi'$.
\end{example}

\begin{example}[Clipping Noise]
  \label{ex:clipping-noise}
  Most measurement devices have a limited output range, i.e., $\Xi'=[a, b]$ with $a<b$, which is a strict subset of all potential outcomes $\Xi=\Re$. In this case we may only observe the censured data
  \(
    \xi_i'' = \max(\min(\xi'_i, b), a)
  \)
  instead of data $\xi'_i$ which itself has been corrupted by Gaussian noise as discussed in Example \ref{ex:gaussian-noise}.
  This observational model is characterized by the map $O^{CN}: \xi \mapsto N(\xi,\sigma^2)[-\infty, a)\cdot \delta_a + N(\xi,\sigma^2)(b, \infty]\cdot \delta_b + N(\xi,\sigma^2)[a, b]\cdot N_{[a, b]}(\xi, \sigma^2)$
  where $N_{[a, b]}(\xi, \sigma^2)$ denotes here a normal distribution truncated to the interval $[a,b]$.
  Assumption \ref{ass:density} holds here with
  \[
    d(\xi, \xi')=
    \begin{cases}
      -\log N(\xi,\sigma^2)[-\infty, a] &  {\rm{if}}~ \xi' = a,\\
      \norm{\xi-\xi'}_2^2/(2\sigma^2) & {\rm{if}}~ \xi'\in (a, b),\\
      -\log N(\xi,\sigma^2)[b, \infty] &  {\rm{if}}~ \xi' = b
    \end{cases}                 
  \]
  for all $\xi\in \Xi$ and $m'=\tfrac{\mu'}{(\sigma \sqrt{(2\pi)^{\dim(\Xi')}})}+\delta_{a} +\delta_{b}$ with $\mu'$ the Lebesgue measure on $\Re$ and $\delta_{a}$ and $\delta_b$ two Dirac measures at locations $a$ and $b$, respectively.
\end{example}

\begin{example}[Quantization Noise]
  \label{ex:quantization-noise}
  Digital measurements quantize the noisy measurements of Example \ref{ex:gaussian-noise} further in the sense that $\Xi'$ is necessarily only a finite subset of $\Xi$.
  Let $\Xi_{\xi'}$ denote the collection of all inputs $\xi$ in $\Xi$ which get quantized into the digital symbol $\xi'\in \Xi'$.
  This observational model is characterized by the map $O^{QN}: \xi \mapsto \sum_{\xi'\in \Xi'} N(\xi,\sigma^2)[\Xi_{\xi'}] \cdot \delta_{\xi'}$.
  Assumption \ref{ass:density} holds here with
  \[
    d(\xi, \xi')= -\log N(\xi, \sigma^2)[\Xi_{\xi'}] 
  \]
  for all $\xi'\in \Xi'$ and $\xi\in \Xi$ with $\mu'$ the counting measure on $\Xi'$.
\end{example}
}

We will attempt to infer the unknown probability measure $P$ from the noisy data based on its empirical probability measure $P_N'$.
  Clearly, considering the empirical probability measure rather than the noisy data directly imposes no loss of information as the order of the data points is of no consequence here.
  Sanov's theorem \citep[Theorem 6.2.10]{dembo2009large} ensures also here that the sufficient statistic $P_N'$ enjoys a large deviation property.
  That is, the empirical probability measure $P'_N$ satisfies for any open subset $\mc O\subseteq \mc P(\Xi')$ the large deviation lower bound
  \begin{subequations}
    \label{eq:ldp_exponential_rates:old}
    \begin{align}
      \label{eq:ldp_exponential_rates_lb:old}
      -\inf_{\hat P' \in \mc O} \, \KL(\hat P', O\star P)~\leq ~&  \varliminf_{N\to \infty}~\frac1N \log \Prob[ P'_N \in \mc O ] \\
      \intertext{and for any closed subset $\mc C\subseteq \mc P(\Xi')$ the large deviation upper bound}
                                                     &  \varlimsup_{N\to \infty}~\frac1N \log \Prob[ P_N'\in \mc C ] \leq -\inf_{\hat P' \in \mc C} \,  \KL(\hat P', O\star P). \label{eq:ldp_exponential_rates_ub:old}
    \end{align}
  \end{subequations}
  for the good\footnote{A rate function $I$ is good if its sublevel sets $\tset{\hat P'\in \mc P(\Xi')}{I(\hat P', P)\leq r}$ for any $r\geq 0$ and $P\in \mc P$ are compact \cite{dembo2009large}.} rate function $I(\hat P', P)\defn \KL(\hat P', O\star P)$.
  We remark that large deviation inequalities generally are quite rough in nature as indeed (\ref{eq:ldp_exponential_rates_lb:old}) and  (\ref{eq:ldp_exponential_rates_ub:old}) only pertain to open or closed sets (in the topology of weak convergence), respectively. The rate function is observed to be nonnegative and in fact $I(\hat P', P')=0$ if and only if $\hat P'=P'$. For any $\epsilon>0$, the large deviation inequality \eqref{eq:ldp_exponential_rates_ub:old} implies
\begin{align*}
  \varlimsup_{N\to \infty}~\frac1N \log \Prob[ \LP{P_N'}{P'}\geq \epsilon ] \leq -\min \set{ I(\hat P', P)}{\LP{\hat P'}{P'}\geq \epsilon} <0 \quad \forall P\in\mc P
\end{align*}
where the minimum is indeed achieved as our good rate function has compact sublevel sets and the set of all $\hat P'\in \mc P(\Xi')$ such that $\LP{\hat P'}{P'}\geq \epsilon$ is by definition closed\footnote{We denote here with $\LP{\cdot}{\cdot}$ the L\'evy-Prokhorov metric \cite{prokhorov1956convergence} which metricizes the weak topology on $\mc P(\Xi')$.} and does not contain the distribution $P'$.
Hence, the large deviation property immediately implies that the empirical probability measure $P_N'$ converges in probability to $P'$ with an increasing number of observations.
In fact, the rate function can be interpreted to quantify the exponential speed with which this convergence in probability takes place.

\subsection{Efficiency}
\label{ssec:efficiency}

We consider a prescriptive problem in which we attempt to learn the solution to the stochastic optimization problem stated in Equation (\ref{eq:sop}) from the noisy observational data described in Section \ref{sec:noisy-data}.
Let us denote with $P^\ml_N$ the maximum likelihood estimate for the unobserved probability distribution $P$ which as pointed out by  \citet{rigollet2018entropic} can under certain conditions be computed efficiently as $P^\ml_N=\arg\min_{P\in \mc P} W_1(P'_N, P)$.
A straightforward extension of the sample average formulation in Equation \eqref{eq:saa} to this noisy data would be to consider
\begin{equation}
  \label{eq:mle}
  z(P^\ml_N) \in \arg_\epsilon\inf_{z\in Z}\, \E{P^\ml_N}{\ell(z, \xi)}.
\end{equation}
Many other formulations based on different distributional estimates are evidently possible as well.
This naturally leads us to question if between these many alternative data-driven formulations one ought to be preferred over the other from a statistical point of view?
To answer this question more broadly we must of course first define precisely what constitutes a data-driven formulation and secondly agree on how its statistical performance should be quantified.
We follow the framework presented in \cite{vanparys2020data} and define a data-driven formulation as consisting of a predictor and prescriptor. 

\begin{definition}[Predictors and prescriptors]
  \label{def:dd_prediction}
  A measurable function~$\tilde c: Z \times\mc P(\Xi')\to\Re$ is called a predictor.
 A measurable function $\tilde z: \mc P(\Xi')\to Z$ is called a prescriptor if there exists a predictor~$\tilde c$ that induces~$\tilde z$ in the sense that $\tilde z(\hat P') \in \arg_\epsilon\inf_{z\in Z} \tilde c(z, \hat P')$ for all $\hat P'\in\mc P(\Xi')$. That is, we have $\tilde c(\tilde z(\hat P'), \hat P')-\epsilon< \tilde v(\hat P')\defn\inf_{z\in Z} \tilde c(z, \hat P')$ where we denote the function $\tilde v:\mc P(\Xi')\to\Re$ as the optimal value function of the formulation.
\end{definition}

The maximum likelihood formulation (\ref{eq:mle}) employs the cost predictor $\E{P^\ml_N}{\ell(z, \xi)}$ to prescribe its decision $z(P^\ml_N)$. However, the maximum likelihood is a mere point estimate of the unobserved probability distribution $P$. The maximum likelihood formulation can consequently be expected to suffer similar shortcomings as the sample average formulation in the noiseless regime.
That is, the cost budgeted for its prescribed decision is likely to disappoint out of sample. Here we say a formulation based on a predictor prescriptor pair $(\tilde c, \tilde z)$ disappoints if the event
\[
  P'_N\in \mc D(\tilde c, \tilde z; P) \defn \set{\hat P'\in \mc P(\Xi')}{c(\tilde z(\hat P'), P) > \tilde c(\tilde z(\hat P'), \hat P')}
\]
occurs with $c(z, P)=\E{P}{\ell(z, \xi)}$ the unknown out-of-sample cost.
Such disappointment events in which the actual cost of our decision, i.e., $c(\tilde z(P'_N), P)$, breaks the predicted cost or budget, i.e., $\tilde c(\tilde z(P'_N), P_N')$, are undesirable and should be avoided by the decision-maker. Consequently, we prefer formulations which keep the disappointment rates
\begin{equation}
  \label{eq:disappointment-rate}
  \varlimsup_{N\to\infty} \,\frac{1}{N} \log \Prob [ P'_N\in \mc D(\tilde c, \tilde z; P) ]
\end{equation}
as small as possible for all $P\in \mc P$.
We will only denote here formulations as feasible if their out-of-sample disappointment probability decays sufficiently fast, i.e., $\eqref{eq:disappointment-rate}\leq -r$.
Evidently, sufficiently fast disappointment probability decay can be achieved trivially by simply inflating the cost budgeted for each decision by some large nonnegative amount.
We would hence prefer those formulations which promise minimal biased long term cost prediction
$\lim_{N\to\infty}\tilde c(\tilde z(P'_N), P'_N)$ for all $P \in \mc P$. 

We consider here the family  of robust formulations defined by the predictor prescriptor pairs
\begin{equation}
  \label{eq:dro-optimal}
  \tilde c^\delta(z, P'_N) \defn  \sup \set{\E{Q}{\ell(z, \xi)}}{Q\in \mc P, ~I^\delta(P'_N, Q)\leq r}, \quad
  \tilde z^\delta(P'_N) \in  \textstyle\arg_\epsilon\inf_{z\in Z}\, \tilde c^\delta(z, P'_N)
\end{equation}
based on our smooth large deviation rate function defined in Equation (\ref{eq:smooth-rate-function}). 
We will show using a large deviation argument that this family dominates the very rich class of regular formulations.

\begin{definition}[Regular predictors and prescriptors]
\label{def:dd_prediction-regular}
A predictor~$\tilde c: Z \times\mc P(\Xi')\to\Re$ is called regular if it is continuous on $Z\times \mc P(\Xi')$. A prescriptor $\tilde z: \mc P(\Xi')\to Z$ is called  regular if it is continuous and there exists a regular predictor~$\tilde c$ that induces~$\tilde z$ in the sense that $\tilde z(\hat P') \in \arg_\epsilon\inf_{z\in Z} \tilde c(z, \hat P')$ for all $\hat P'\in\mc P(\Xi')$.
\end{definition}

{\color{black} For regular predictors we have that the observed random cost $\tilde c(\tilde z(P'_N), P'_N)$ converges almost surely to $\tilde c(\tilde z(O\star P), O\star P) $ as the empirical distribution $P_N'$ converges almost surely to $O\star P$ following \cite[Theorem 11.4.1]{dudley2018real} for every $P\in \mc P$.}
Remark that the class of all regular formulations is very rich as Definition \ref{def:dd_prediction-regular} imposes only mild structural restrictions. 
The Berge maximum theorem \cite[p.\ 116]{berge1997topological} indeed implies that the optimal value function $\tilde v(\hat P') = \min_{z\in Z} \tilde c(z, \hat P')$ of any regular formulation is a continuous function on $\mc P'(\Xi)$ already when the constraint set $Z$ is merely compact.
The correspondence $\hat P'\mapsto \tset{z\in Z}{\tilde c(z, \hat P')< \tilde v(\hat P')+\epsilon}$ of $\epsilon$-suboptimal solutions in a regular formulation is consequently lower semicontinuous \cite[Corollary 4.2.4.1]{bank1982non} for any $\epsilon>0$.
Hence, for formulations employing a convex predictor $\tilde c$ and $\mc P(\Xi')$ a compact set, an associated regular predictor can always be found \cite[Theorem 9.1.]{aubin2009set}.
Should a regular formulation admit unique optimal decisions, such decisions will constitute a regular prescriptor as well following \cite[p.\ 117]{berge1997topological}.
The need to focus on this restricted but nevertheless quite rich class of regular formulations is necessary due to the rough nature of the employed large deviation argument.

\begin{assumption}
  \label{ass:dro-uc}
  The cost function $c:Z\times\mc P\to\Re, ~(z, P)\mapsto \E{P}{\ell(z, \xi)}$ is continuous.
\end{assumption}

We remark that Assumption \ref{ass:dro-uc} is rather mild and is already satisfied when the loss function $\ell : Z\times \Xi\to\Re$ is merely bounded and uniformly continuous.

\begin{theorem}
  \label{cor:optimal-dro}
  Let Assumption \ref{ass:dro-uc} hold. Then,
  the family of predictor prescriptor pairs $( \tilde c^\delta , \tilde z^\delta)$ is feasible for any $\delta>0$, i.e.,
  \begin{equation}
    \label{eq:dro:feasibility}
    \varlimsup_{N\to\infty} \frac 1N \log \Prob \left[P'_N\in \mc D(\tilde c^\delta , \tilde z^\delta; P )\right] \leq -r \quad \forall P\in \mc P.
  \end{equation}
  Furthermore, consider any regular predictor prescriptor pair $( \tilde c , \tilde z)$ which satisfies
  \begin{equation}
    \label{eq:dro:feasibility-inflated}
    \varlimsup_{N\to\infty} \frac 1N \log \Prob \left[P'_N\in \mc D(\tilde c , \tilde z; P )\right] \leq -r ~~~\quad \forall P\in \mc P.
  \end{equation}
  Then, we have that for all $\epsilon>0$ there exists $0<\delta'$ so that any $0<\delta\leq \delta'$ we have almost surely
  {\color{black}
  \begin{equation}
    \label{eq:dro:optimality-inflated}
    \lim_{N\to\infty} \tilde c(\tilde z(P_N'), P_N') +3\epsilon \geq  \lim_{N\to\infty} \tilde c^\delta(\tilde z^\delta(P_N'), P_N')
  \end{equation}}
  when the noisy data is generated by any distribution $P'\in \mc P'$.
 \end{theorem}
\begin{proof}
  We start by proving that our family of formulations is feasible and satisfies guarantee (\ref{eq:dro:feasibility}).
  To this end define the~sets
  \[
    \mc D^\delta(P)= \{\hat P'\in\mc P(\Xi') : \textstyle\sup_{z\in Z} c(z,P)-\tilde c^\delta(z,\hat P')>0\}\quad\text{and}\quad \mc B^\delta(P)=\{\hat P'\in\mc P(\Xi') : I^\delta(\hat P',P)>r\}.
  \]
  We may assume without loss of generality that we have $\mc D^\delta(P)\neq \emptyset$ for otherwise the out-of-sample disappointment~$\Prob[P'_N\in  \mc D(\tilde c^\delta , \tilde z^\delta; P ) \subseteq \mc D^\delta(P)]$ clearly vanishes for all~$N\geq 1$ and inequality (\ref{eq:dro:feasibility}) holds trivially.
  We will now show that~$\mc D^\delta(P)\subseteq \mc B^\delta(P)$. For the sake of contradiction, choose any~$\hat P'\in \mc D^\delta(P)$, and assume that~$I^\delta(\hat P',P)\leq r$. Thus, we have for some $z\in Z$ that
  \begin{equation*}
    c(z,P)> \tilde c^\delta(z,\hat P')=  \sup \set{ c(z, Q) }{ Q \in \mc P,~ I^\delta(\hat P', Q)\leq r } \geq c(z,P);
  \end{equation*}
a contradiction.
  As~$\hat P'\in \mc D^\delta(P)$ was chosen arbitrarily, we have thus shown that~$\mc D^\delta(P)\subseteq \mc B^\delta(P)$ and hence also $\cl \mc D^\delta(P)\subseteq \cl \mc B^\delta(P)$. Next we show that $\hat P'\in \cl \mc B^\delta(P) \implies I(\hat P', P)> r$.
  For the sake of contradiction assume that we have found $\hat P'\in \cl \mc B^\delta(P)$ for which $I(\hat P', P)\leq r$. There must exist a sequence $\hat P'_k\in \mc B^\delta(P)$ which converge to $\hat P'$ and hence $\LP{\hat P'_k}{\hat P'}$ tends to zero.
  However, from the definition of $\mc B^\delta(P)$ we have that in fact for all $Q'\in \mc P(\Xi')$ such that $\LP{Q'}{\hat P_k'}\leq \delta$ we have that $I(Q', P)> r$.
  Take now $k$ large enough such that $\LP{\hat P'_k}{\hat P'}  \leq \delta$ then we must have $I(\hat P', P)> r$; a contradiction.
  The above reasoning implies that
  \begin{align*}
    & \varlimsup_{N\to\infty}\frac 1N \log \Prob\left[ c(\tilde z^\delta(P_N'),P) > \tilde c^\delta(\tilde z^\delta(P'_N),P'_N) \right] \\
    &\hspace{2cm}\leq  \varlimsup_{N\to\infty}\frac 1N \log \Prob \left[ \sup_{z\in Z} c(z,P) - \tilde c^{\delta}(z,P_N') >0 \right] \leq - \inf_{\hat P'\in\cl \mc B^\delta(P)} ~ I(\hat P',P)  \leq -r
  \end{align*}
  establishing inequality (\ref{eq:dro:feasibility}) as $P\in \mc P$ was arbitrary.
  
  We will prove that for any regular predictor prescriptor pair $(\tilde c, \tilde z)$ which satisfies guarantee \eqref{eq:dro:feasibility-inflated}  we have
  \begin{equation}
    \label{eq:dro-uniform-inequality}
    \lim_{\delta\to 0} \inf_{P'\in \mc P'}\tilde c(\tilde z(P'), P') - \tilde c^\delta(\tilde z(P'), P') \geq 0.
  \end{equation}
  From inequality (\ref{eq:dro-uniform-inequality}) we can take $\delta'' > 0$ sufficiently small so that 
  \(
  \tilde c(\tilde z(P'), P') \geq \tilde c^{\delta''}(\tilde z(P'), P')-\epsilon\geq \tilde c^{\delta''}(\tilde z^{\delta''}(P'), P')-2\epsilon
  \)
  uniformly for all $P'\in \mc P'=\set{O\star P}{P\in\mc P}$. 
  {\color{black}Hence, 
    \[
      \tilde c(\tilde z(O\star P), O\star P)+ 2 \epsilon \geq \tilde c^{\delta''}(\tilde z^{\delta''}(O\star P), O\star P)  \quad \forall P\in \mc P.
    \]
    Assume now that $\LP{P_N'}{O\star P}\leq \delta' = \delta''/2$ occurs.
    The triangular inequality guarantees that $\LP{Q'}{P_N'}=\LP{P_N'}{Q'}\leq \delta'$, $\LP{P_N'}{O\star P}\leq \delta'$ implies $\LP{O\star P}{Q'}\leq 2\delta'=\delta''$.
    Hence, we have that
    $\tilde c^{\delta''}(\tilde z^{\delta''}(O\star P), O\star P) = \sup \tset{c(\tilde z^{\delta''}(O\star P), Q)}{Q\in \mc P,~ Q'\in \mc P(\Xi'),~\LP{O\star P}{Q'}\leq \delta'',~I(Q', Q)\leq r} \geq \sup \tset{c(\tilde z^{\delta''}(O\star P), Q)}{Q\in \mc P,~ Q'\in \mc P(\Xi'),~\LP{P'_N}{Q'}\leq \delta',~I(Q', Q)\leq r}\geq \tilde c^{\delta'}(\tilde z^{\delta'}(P'_N), P'_N)-\epsilon$. From \citet[Theorem 11.4.1]{dudley2018real} we know that $\lim_{N\to\infty}\LP{P_N'}{O\star P}=0$ almost surely and hence $\tilde c^{\delta''}(\tilde z^{\delta''}(O\star P), O\star P) \geq \lim_{N\to\infty} \tilde c^{\delta'}(\tilde z^{\delta'}(P'_N), P'_N)-\epsilon$ almost surely. Hence, we have almost surely 
    \[
      \lim_{N\to\infty} \tilde c(\tilde z(P_N'), P_N') + 2\epsilon = \tilde c(\tilde z(O\star P), O\star P)+ 2 \epsilon \geq \tilde c^{\delta''}(\tilde z^{\delta''}(O\star P), O\star P) \geq \lim_{N\to\infty} \tilde c^{\delta'}(\tilde z^{\delta'}(P'_N), P'_N)-\epsilon.
    \]
    The theorem now follows by remarking that $\tilde c^{\delta}(z, P')\leq \tilde c^{\delta'}(z, P')$ for all $z\in Z$, $P'\in \mc P(\Xi')$ and $0<\delta\leq \delta'$.
  }

  We will now establish inequality (\ref{eq:dro-uniform-inequality}) by showing that
  \begin{equation}
    \label{eq:dro-uniform-inequality-rho}
    \lim_{\delta\to 0} \inf_{P'\in \mc P'}\tilde c(\tilde z(P'), P') - \tilde c^\delta(\tilde z(P'), P') \geq -3\rho
  \end{equation}
  for any $\rho>0$.
  Assume for the sake of contradiction that $\lim_{\delta\to 0} \inf_{P'\in \mc P'}\tilde c(\tilde z(P'), P') - \tilde c^\delta(\tilde z(P'), P') < -3\rho$.
  There must hence exist a $\delta'>0$ such that for all $\delta\leq \delta'$ we have $\inf_{P'\in \mc P'}\tilde c(\tilde z(P'), P') - \tilde c^\delta(\tilde z(P'), P') < -2\rho$.
  Consequently, there exists a distribution $Q''(\delta)\in \mc P'$ such that
  \begin{align*}
    \tilde c(\tilde z(Q''(\delta)), Q''(\delta)) + 2\rho  < & ~\tilde c^\delta(\tilde z(Q''(\delta)), Q''(\delta)) \quad \forall \delta\leq \delta' \\[0.5em] 
     < & ~\sup \set{c(\tilde z(Q''(\delta)), P)}{P\in \mc P, \, I^\delta(Q''(\delta), P)\leq r} \quad \forall \delta\leq \delta'.
  \end{align*}
  Hence, there  exists for all $\delta\leq \delta'$ a distribution $P^\star(\delta)\in \mc P$ such that $I^\delta(Q''(\delta), P^\star(\delta)) \leq r$ and $\tilde c(\tilde z(Q''(\delta)), Q''(\delta)) + 2\rho < c(\tilde z(Q''(\delta)), P^\star(\delta))$.
  From the definition of the smooth rate function $I^\delta$ stated in Equation (\ref{eq:smooth-rate-function}) this implies that there exists an auxiliary sequence $Q'''(\delta)\in \mc P(\Xi')$ such that $I(Q'''(\delta), P^\star(\delta)) \leq r$ with $\LP{Q'''(\delta)}{Q''(\delta)}\leq \delta$ for all $\delta\leq \delta'$. Remark that
  \begin{align*}
    & \lim_{\delta\to 0}\tilde c(\tilde z(Q'''(\delta)), Q'''(\delta)) = \lim_{\delta\to 0}\tilde c(\tilde z(Q''(\delta)), Q''(\delta)),\\[0.5em]
    & \lim_{\delta\to 0} c(\tilde z(Q'''(\delta)), P^\star(\delta)) = \lim_{\delta\to 0} c(\tilde z(Q''(\delta)), P^\star(\delta))
  \end{align*}
  as the functions $c:Z\times \mc P \to \Re$, $\tilde c:Z\times \mc P(\Xi')\to\Re$ and $\tilde z:\mc P(\Xi')\to Z$ are continuous and $\LP{Q'''(\delta)}{Q''(\delta)}\leq \delta$ implies $\lim_{\delta\to 0}Q'''(\delta)=\lim_{\delta\to 0}Q''(\delta)$.
  Consequently, there exists a $\delta^\star \in (0, \delta']$ so that with $P^\star=P^\star(\delta^\star) \in \mc P$ and $Q^\star=Q'''(\delta^\star)\in \mc P(\Xi')$ we have both $$\abs{\tilde c(\tilde z(Q''(\delta^\star)), Q''(\delta^\star))-\tilde c(\tilde z(Q^\star), Q^\star)}< \rho ~ {\rm{and}} ~ \abs{c(\tilde z(Q''(\delta^\star)), P^\star) -c(\tilde z(Q^\star), P^\star)} < \rho.$$
  Hence, we have both $\tilde c(\tilde z(Q^\star), Q^\star) < c(\tilde z(Q^\star), P^\star)$ and $I(Q^\star, P^\star) \leq r$. Define the continuous function $Q': [0, 1]\to \mc P(\Xi'), ~\lambda \mapsto O\star P^\star\cdot \lambda + Q^\star\cdot (1-\lambda)$ and recall that $I(O\star P^\star, P^\star)=0$.
  As we have that the functions $c:Z\times \mc P \to \Re$, $\tilde c:Z\times \mc P(\Xi')\to\Re$ and $\tilde z:\mc P(\Xi')\to Z$ are continuous there hence exists $Q' = Q'(\lambda')$ for $\lambda' \in (0, 1]$ sufficiently small so that $I(Q', P^\star) \leq \lambda' I(O\star P^\star, P^\star) + (1-\lambda') I(Q^\star, P^\star) \leq (1-\lambda') \cdot r = r'   < r$ using the convexity of the rate function and $\tilde c(\tilde z(Q'), Q') < c(\tilde z(Q'), P^\star)$ where we use Assumption \ref{ass:dro-uc}.
  Consequently, we have that 
  \(
    Q'\in \mc D(\tilde c, \tilde x; P^\star)= \interior\mc D(\tilde c, \tilde x; P^\star)
  \)
  as again we remark for a final time that the functions $c:Z\times \mc P \to \Re$, $\tilde c:Z\times \mc P(\Xi')\to\Re$ and $\tilde z:\mc P(\Xi')\to Z$ are continuous. From the large deviation inequality (\ref{eq:ldp_exponential_rates_lb:old}) it follows now that 
  \[
    -r' \leq -I(Q', P^\star) <-\inf_{\hat P'\in \interior \mc D(\tilde c, \tilde x; P^\star)} I(\hat P', P^\star) \leq \varliminf_{N\to\infty} \frac 1N \log \Prob \left[P'_N\in \mc D(\tilde c , \tilde z; P^\star )\right]
  \]
  which is in direct contradiction with the feasibility inequality (\ref{eq:dro:feasibility-inflated}) as $r'<r$ which establishes our inequality (\ref{eq:dro-uniform-inequality-rho}) as $\rho>0$ was arbitrary.
\end{proof}

Inequality (\ref{eq:dro:feasibility}) guarantees that any formulation in our family is feasible.
Inequalities (\ref{eq:dro:feasibility-inflated}) and (\ref{eq:dro:optimality-inflated}) guarantee that any other feasible regular formulation is dominated (modulo the small constant $3\epsilon>0$) by members of our efficient family for all parameters $0<\delta\leq \delta'$ with $0<\delta'$ sufficiently small.
The previous theorem hence indicates that our family dominates any regular formulation in terms of balancing the desire for small out-of-sample disappointment as well as minimal bias under Assumption \ref{ass:dro-uc} and are thus {efficient}.
Additionally, we remark here that our focus on the asymptotic guarantees (\ref{eq:dro:feasibility}) and (\ref{eq:dro:feasibility-inflated}) and long term out-of-sample costs allows us to consider a much stronger sense of efficiency than classical minimax efficiency \citep{lehmann2006theory}.
{\color{black}Minimax efficiency would amount here to guaranteeing that
\(
  \lim_{N\to\infty} \tilde c(\tilde z(P_N'), P_N') - \lim_{N\to\infty} c^\delta(\tilde z^\delta(P_N'), P_N')\geq 0
\)
when the noisy data is generated by \textit{some distribution} $P'\in \mc P$.
Guarantee (\ref{eq:dro:optimality-inflated}) is much stronger in that we have
\(
\lim_{N\to\infty} \tilde c(\tilde z(P_N'), P_N') - \lim_{N\to\infty} c^\delta(\tilde z^\delta(P_N'), P_N')\geq -3\epsilon
\)
for \textit{any distribution} $P'\in \mc P$.}

In view of the previous discussion it is tempting to consider the data-driven formulation with predictor and prescriptor
\begin{equation}
  \label{eq:dro-zero}
  \tilde c^0(z, P'_N) \defn  \sup \set{\E{P}{\ell(z, \xi)}}{Q\in \mc P, ~I(P'_N, Q)\leq r}, \quad
  \tilde z^0(P'_N) \in  \arg\min_{z\in Z}\, \tilde c^0(z, P'_N)
\end{equation}
based directly on our rate function $I$ rather than its $\delta$-smoothed counterpart $I^\delta$. Van Parys et al.\ \cite{vanparys2020data} prove indeed that when given access to the noiseless data in Equation (\ref{eq:noiseless-observations}) with empirical distribution $P_N$ the appropriate rate function is precisely the entropic distance $\KL(P_N, P)$ and formulation \eqref{eq:dro-zero} is efficient in the noiseless setting as discussed in Section \ref{sec:introduction}.
However, recall again that for noisy data when the base measure $m'$ defined in Assumption \ref{ass:density} fails to be atomic, the ambiguity set $\set{Q\in \mc P}{I(P_N', Q) <\infty} = \emptyset$ is trivial for any empirical distribution $P_N'$ and consequently the associated data-driven formulation (\ref{eq:dro-zero}) is here obviously infeasible. Hence, considering a somewhat smoothed rate function $I^\delta$ instead of the rate function $I$ directly seems unavoidable when faced with noisy observational data.

We conclude here by providing an interesting connection between the actual rate function $I$ and the entropic optimal transport which sheds a light on the role entropic optimal transport plays in this noisy observational data regime. We defer the proof of the next result to Appendix \ref{sec:proof-theorem-ldp}.

\begin{theorem}
  \label{thm:ldp}
  Let Assumption \ref{ass:density} hold. The jointly convex rate function can be decomposed as
  \begin{equation}
    \label{eq:infimum_over_Q}
    I(\hat P', P) = \inf_{Q\in \mc P(\Xi)}\KL(Q, P) + \KL(\hat P', m') +  {W_1}(\hat P', Q) \geq 0.
  \end{equation}
\end{theorem}

In the proof of Theorem \ref{thm:ldp} we indicate that if $\hat P' \ll m'$ the infimum in Problem \eqref{eq:infimum_over_Q} is achieved at $Q^\star$ whose Radon-Nikodym derivative with respect to $P$ is 
\(
\tfrac{\d Q^\star}{\d P}(\xi) = \int_{\Xi'} [\tfrac{\exp(-d(\xi, \xi'))}{\int_{\Xi} \exp(-d (\xi'', \xi')) \d P(\xi'')}] \d \hat P'(\xi')
\)
for all $\xi\in \Xi$. 
The optimization variable $Q$ in Theorem \ref{thm:ldp} can be interpreted to represent the unobserved empirical distribution $P_N$ of the noiseless data points. With this interpretation in mind the first term $\KL(Q, P)$ ensures that $\Prob[P_N\approx Q]\asymp \exp(-N \cdot \KL(Q, P))$ following Sanov's theorem and accounts for the fact that the empirical distribution $P_N$ of the noiseless data may differ from unknown the distribution $P$ when the number of training data points is finite. The last two terms accounts for the fact that we only observe the empirical distribution $P'_N$ of the noisy data. One can show that this term quantifies indeed $\Prob[P'_N\approx \hat P'|P_N\approx Q] \asymp \exp(-N \cdot (\KL(\hat P', m')+{W_1}(\hat P', Q)))$.
Informally, we have using the law of total probability that
\begin{align*}
  \Prob[P'_N\approx \hat P'] = & \int \Prob[P'_N\approx \hat P'|P_N\approx Q] ~ \Prob[P_N\approx Q]\\[0.5em]
  \asymp & \int \exp(-N \cdot (\KL(\hat P', m')+{W_1}(\hat P', Q))) \exp(-N \cdot \KL(Q, P)) \\[0.5em]
  \asymp & \exp(-N \cdot \inf_{Q\in \mc P} \KL(Q, P) + [\KL(\hat P', m') +  {W_1}(\hat P', Q)] ).
\end{align*}
We remark that the entropic optimal transport term ${W_1}(\hat P', Q)$ is defined in Equation (\ref{def:entropic-wasserstein}) where its marginal transport cost $d$ is identified here as the logarithm of the density function of the noise corrupting our data as indicated in Assumption \ref{ass:density}. The term $\KL(\hat P', m')$ can be interpreted as a compensation term for the fact that any density function arbitrarily depends on the base measure $m'$ considered.

We remark that the optimization problem in Equation (\ref{eq:infimum_over_Q}) is identified as a so called \textit{unbalanced} optimal transport problem initially proposed by \cite{benamou2003numerical}. 
\citet{chizat2018scaling} propose scalable algorithms based on an entropic regularization scheme for such unbalanced optimal transport problems which have recently witnessed a flurry of interest, c.f., \citet{sejourne2019sinkhorn} and references therein.

{\color{black}
  \subsection{Consistency}
  \label{sec:consistency}

Strong out-of-sample guarantees such as those we impose in Equation (\ref{eq:dro:feasibility}) yield conservative formulations even when a large amount of data points are observed.
Indeed, even in the noiseless regime \citet{vanparys2020data} point out that imposing the out-of-sample guarantee (\ref{eq:dissapoint-gen}) with $r>0$ necessarily leaves any feasible formulation to be inconsistent, i.e.,
$\E{P}{\ell(z_{\KL, r}(P_N), \xi)}$ does not necessarily converge to $\inf_{z}\E{P}{\ell(z, \xi)}$.
Intuitively, this is a direct result from the fact that the employed ambiguity set $\set{Q\in \mc P}{\KL(P_N, Q)\leq r}$ does not reduce to $\{P\}$ when $N$ is large as the robustness radius $r$ is here constant.
Nevertheless, \citet{lam2019recovering, duchi2016statistics} show that if the probability of disappointment is only required to be remain bounded rather than to decay exponentially as in Equation (\ref{eq:dissapoint-kl}) a consistent formulation can be derived by simply reducing the robustness radius $r_N$ with increasing $N$ under mild technical conditions on the cost function $\ell$.

We remark here that by imposing the strong out-of-sample guarantee in Equation (\ref{eq:dro:feasibility}), our family of efficient formulations is also inconsistent as their associated ambiguity sets $\set{Q\in \mc P}{I^\delta(P'_N, Q)\leq r}$ do not shrink to $\{P\}$ even when $N$ is large as both here both $r$ as well as $\delta$ are constant. Motivated by the previous discussion we consider simply reducing $r_N$ and $\delta_N$ as $N$ grows and consider the robust formulation
\begin{equation}
  \label{eq:sample-dependent-dro}
  \tilde c_N(z, P'_N) \defn  \sup \set{\E{Q}{\ell(z, \xi)}}{Q\in \mc P, ~I^{\delta_N}(P'_N, Q)\leq r_N}, \quad \tilde z_N(P'_N) \in  \textstyle\arg_\epsilon\inf_{z\in Z}\, \tilde c_N(z, P'_N).
\end{equation}
We first show that the ambiguity set associated with the previously introduced formulation contains indeed $P$ with high probability as the number of observations is large and both $r_N$ and $\delta_N$ are reduced at appropriate rates with increasing $N$.

\begin{proposition}
  \label{prop:confidence}
  Assume that $P'=O\star P$ has bounded fourth moment and consider two nonincreasing sequences $r_N>0$ and $\delta_N>0$ with $r_N = \Omega(N^{\gamma-1})$ and $\delta_N = \Omega(N^{-\tfrac{\gamma'}{(2\dim(\Xi'))}})$ for some $0<\gamma'<\gamma<1$. 
  Then, 
  \[
    \lim_{N\to\infty} \Prob\left[\exists Q\in \mc P(\Xi') ~\st~ \LP{P_N'}{Q}\leq \delta_N,~I(Q, P)\leq r_N\right] = 1.
  \]
\end{proposition}

When there exists $Q\in \mc P(\Xi')$  so that $\LP{P_N'}{Q}\leq \delta_N$ and $I(Q, P)\leq r_N$, the distribution $P$ must be contained in the ambiguity set of the predictor $\tilde c_N$ which by definition implies that $c_N(z, P'_N)>\E{P}{\ell(z, \xi)}$ for any $z\in Z$.
The previous proposition hence immediately implies that 
\(
  \Prob \left[\E{P}{\ell(\tilde z_N(P'_N), \xi)} > \tilde c_N(\tilde z_N(P'_N), P'_N) \right]
\)
decays to zero, however, not necessarily exponentially fast. 

\begin{remark}[Finite Sample Guarantees]
  \label{rem:finite-sample-guarantees}
  We remark that Equation \eqref{eq:finite-sample-bound} in the proof of Proposition \ref{prop:confidence} provides also a finite sample guarantee on the probability of disappointment $\Prob \left[\E{P}{\ell(\tilde z_N(P'_N), \xi)} > \tilde c_N(\tilde z_N(P'_N), P'_N) \right]$. This enables the construction for any desired maximum disappointment probability $\beta\in (0,1]$ of two particular sequences $r^\star_N>0$ and $\delta^\star_N>0$ tending to zero at an appropriate speed (with in fact $r^\star_N = \mc O(N^{\gamma-1})$ and $\delta^\star_N = \mc O(N^{-\tfrac{\gamma'}{(2\dim(\Xi'))}})$ for some $0<\gamma'<\gamma<1$) so that $\Prob \left[\E{P}{\ell(\tilde z_N(P'_N), \xi)} > \tilde c_N(\tilde z_N(P'_N), P'_N) \right]\leq \beta$ for all $N\geq 1$.
  Furthermore, we remark that our efficiency guarantee in Theorem \ref{thm:ldp} on the formulation proposed in Equation \eqref{eq:dro-optimal} is asymptotic in nature. The result does hence not say after what number of samples our asymptotic guarantees are supposed to kick in. In particular, for small $\delta>0$, the associated ambiguity set $\set{Q\in \mc P}{I^\delta(P_N, Q)\leq r}$ may be empty and hence the formulation disappoints as by convention here $\E{P}{\ell(\tilde z^\delta(P'_N), \xi)} > \tilde c^\delta(\tilde z^\delta(P'_N), P'_N) = -\infty$.  If finite sample guarantees are a concern, we propose the robust formulation in Equation \eqref{eq:sample-dependent-dro} with $\delta_N = \delta/2+\delta^\star_N$ and $r_N=r+r^\star_N$ which trivially satisfies both the finite sample guarantee $$\Prob \left[\E{P}{\ell(\tilde z_N(P'_N), \xi)} > \tilde c_N(\tilde z_N(P'_N), P'_N) \right]\leq \beta \quad \forall N\geq 1$$ as $\tilde c_N(\tilde z_N(P'_N), P'_N)\geq \tilde c^{\delta/2}(\tilde z^{\delta/2}(P'_N), P'_N)$ and the asymptotic guarantee $$\varlimsup_{N\to \infty}\frac 1N  \log \Prob \left[\E{P}{\ell(\tilde z_N(P'_N), \xi)} > \tilde c_N(\tilde z_N(P'_N), P'_N) \right]\leq -r$$ as $\delta_N\geq \delta_N^\star$ and $r_N\geq r^\star_N$. Furthermore, the proposed formulation does not asymptotically impose any additional conservatism as indeed
  \(
  \lim_{N\to\infty} \tilde c^\delta(\tilde z^\delta(P_N'), P_N') \geq \lim_{N\to\infty} \tilde c_N(\tilde z_N(P_N'), P_N')
  \)
  due to $\lim_{N\to\infty}\delta_N = \delta/2<\delta$.
\end{remark}

When there exists $Q\in \mc P(\Xi')$  so that $\LP{P_N'}{Q}\leq \delta_N$ and $I(Q, P)=\KL(Q, P')\leq r_N$, in fact all distributions in the set $\set{Q\in \mc P}{O\star Q=P'}$ are contained in the ambiguity set of the predictor $\tilde c_N$. Proposition \ref{prop:confidence} establishes that with high probability the ambiguity set will not only contain $P$ but in fact all distributions in the set $\set{Q\in \mc P}{O\star Q=P'}$.
Unsurprisingly, consistency demands at a bare minimum that the distribution $P$ is identifiable from $P'$, i.e., we must impose $\set{Q\in \mc P}{O\star Q=P'}=\{P\}$ on the observational model $O$.
Not all observational models satisfy this assumption as certain types of measurement noise may lead to information loss. In an extreme case where $O_{\xi}=P'$ independent of $\xi$ then clearly all information regarding the noiseless data is lost and in fact $\set{Q\in \mc P}{O\star Q=P'}=\mc P$.
We remark that our efficiency notion is independent of identifiability. We prove here consistency for the additive error setting under a mild assumption on the loss function $\ell$ which is slightly stronger than the condition imposed in Assumption \ref{ass:dro-uc}.

\begin{assumption}[Identifiability]
  \label{ass:identifiability}
  Let $\varphi_E: \Re^{\dim(\Xi)}\to \mathrm C$,
\(
t \mapsto \int \exp(i\iprod{t}{e}) \,\d E(e) 
\)
be the characteristic function of the error distribution $E$.
 Assume that $\varphi_E$ has no roots, i.e., $\varphi_E(t)\neq 0$ for $t\in \Re^{\dim(\Xi)}$.
\end{assumption}

\begin{assumption}[Bounded Lipschitz Loss]
  \label{ass:bounded-lipschitz}
  Assume that the loss function $\ell$ is uniformly bounded and Lipschitz, i.e., we have $L>0$ and $\mc L>0$ so that
  $\sup_{\xi\in \Xi} \abs{\ell(z, \xi)}\leq L$ as well as $\sup_{z\in Z, \xi_1\neq\xi_2\in \Xi}\tfrac{\abs{\ell(z, \xi_1)-\ell(z, \xi_2)}}{\norm{\xi_1-\xi_2}}\leq \mc L$ for all $z\in Z$.
\end{assumption}

Recall that any distribution in $\mc P(\Re^{\dim(\Xi)})$ is uniquely determined by its characteristic function and that the characteristic function of a convolution between two distributions is given as the product of their characteristic functions \citep[Chapter 6]{karr1993probability}. Assumption \ref{ass:identifiability} guarantees that from the noisy distribution $P'=O^{AE}\star P$ we can identify the noiseless distribution $P$ via its characteristic function $\varphi_{P} = \varphi_{P'}/\varphi_E$.
Remark that the characteristic function of a zero mean normal distribution with variance $\sigma^2$ is given as
\(
\varphi_{N(0, \sigma^2 I)}(t) = \exp(-\norm{\sigma t}_2^2/2)>0
\)
and hence Example \ref{ex:gaussian-noise} satisfies Assumption \ref{ass:identifiability}.

\begin{theorem}[Consistency]
  \label{thm:consistency}
  Consider an additive error model $O^{AE}$ for which Assumption \ref{ass:identifiability} is satisfied.
  Assume that $P'=O^{AE}\star P$ has bounded fourth order moment, the loss $\ell$ satisfies Assumption \ref{ass:bounded-lipschitz} and assume that $r_N$ and $\delta_N$ decay to zero with $r_N = \Omega(N^{\gamma-1})$ and $\delta_N = \Omega(N^{-2\gamma'/\dim(\Xi')})$ for some $0<\gamma'<\gamma<1$.
  Then,
  \[
    \lim_{N\to\infty} \Prob\left[\E{P}{\ell(\tilde z_N(P'_N), \xi)}\leq \inf_{z\in Z}~\E{P}{\ell(z, \xi)} + 2\epsilon\right]=1.
  \]
\end{theorem}
}

\section{Finite Formulations}
\label{sec:tract-pred-form}

Our family of efficient robust formulations in Equation \eqref{eq:dro-optimal} is stated in terms of a saddle-point problem which may be difficult to solve in general. Indeed, even the original stochastic optimization problem \eqref{eq:sop} may not be easy to solve.
For the sake of simplicity we assume here that nothing is known about $P$ and hence $\mc P=\mc P(\Xi)$. We remark though that the presented analysis generalizes to the case where $\mc P$ is a convex subset of $\mc P(\Xi)$ with only minor modifications.
If the loss function $\ell(z, \xi)$ is convex in the decision variable $z$ for any $\xi$ then the robust formulations in Equation \eqref{eq:dro-optimal} only require the solution of a convex optimization problem in the decision variable. Whether or not the convex optimization problem characterizing the optimal decision is tractable depends on whether the prediction function $\tilde c^\delta(z, P'_N)$ can be evaluated efficiently which we will discuss now in more depth.

The maximization problem characterizing the prediction function $\tilde c^\delta(z, P'_N)$ in Equation (\ref{eq:dro-optimal}) is convex in the distribution $P$.
Indeed, the closed metric balls
\(
  \tset{\hat P'\in \mc P(\Xi')}{\LP{\hat P'}{P'_N}\leq \delta} \defn \tset{\hat P'\in \mc P(\Xi')}{P'_N(B)\leq \hat P'(B^\delta)+\delta~~ \forall B\subseteq \Xi'}
\)
induced by any metric, and hence in particular the L\'evy-Prokhorov distance, are convex. Consequently, we have that the prediction function is characterized here as
\begin{equation}
  \label{eq:naive-representation}
  \begin{array}{rl}
    \tilde c^\delta(z, P'_N) = \sup & \E{Q}{\ell(z, \xi)}\\[0.3em]
    \st & Q\in \mc P(\Xi), ~\hat P' \in \mc P(\Xi'),\\[0.3em]
        & I(\hat P', Q)\leq r, ~\LP{\hat P'}{P'_N}\leq \delta.
  \end{array}
\end{equation}
However, even in the case where both event sets $\Xi$ and $\Xi'$ have finite cardinality the terminal L\'evy-Prokhorov constraint is characterized using $2^{\abs{\Xi'}}$ linear inequalities which becomes prohibitive even for moderately sized event sets $\Xi'$. This observation may perhaps explain why the L\'evy-Prokhorov metric has not been considered extensively in the robust optimization literature. The sole exception is the work of \citet{erdougan2006ambiguous} who proposes a stochastic approximation method for chance constrained optimization problems with L\'evy-Prokhorov ambiguity.

\subsection{Strassen Representation}
\label{sec:strass-repr}

Surprisingly, by exploiting a classical representation result of \citet{strassen1965existence} the L\'evy-Prokhorov metric  need not result in intractable formulations. To the best of our knowledge, the application of the Strassen representation to derive a tractable reformulation of the L\'evy-Prokhorov distance is novel. 

\begin{theorem}[Strassen Representation]
  \label{thm:strassen-representation}
  Let $\Xi'_N={\rm{supp}}(P'_N)$ denote the support of $P'_N$. Then,
  \begin{equation}
    \label{eq:strassen-representation}
    \begin{array}{rl@{\hspace{7em}}l}
      \tilde c^\delta(z, P'_N) = \sup & \E{Q}{\ell(z, \xi)} & \\[0.3em]
      \st & Q\in \mc P(\Xi), ~\hat P' \in \mc M_+(\Xi'), ~T\in \mc M_+(\Xi'\times\Xi_N'), & \\[0.3em]
                                      & I(\hat P', Q)\leq r, & [{\rm{DV}}:~ \beta \geq 0] \\[0.3em]
                                      & \Pi_{\Xi'} T=\hat P', & [{\rm{DV}}:~ u:\Xi'\to\Re]\\[0.3em]
                                      & \Pi_{\Xi'_N}T=P'_N, & [{\rm{DV}}:~ v:\Xi'_N\to\Re]\\[0.3em]
                                      & \int_{\Xi'\times\Xi'_N}\one{\norm{\xi'-\xi'_i}\leq \delta}\d T(\xi', \xi_i')\geq 1-\delta. &  [{\rm{DV}}:~ \gamma\geq 0]
    \end{array}
  \end{equation}
\end{theorem}
\begin{proof}
  As we assume that $\int_{\Xi'} P'_N(\xi')=1$, we have that
  $T \in \mc M_+(\Xi'\times\Xi_N'),~\Pi_{\Xi'_N}T=P'_N \implies T\in \mc P(\Xi'\times\Xi_N')$. Similarly, $\hat P' \in \mc M_+(\Xi'),~T\in \mc M_+(\Xi'\times\Xi_N'),~\Pi_{\Xi'_N}T=P'_N,~ \Pi_{\Xi'} T=\hat P' \implies \hat P'\in \mc P(\Xi'_N)$.
  The rest of the statement is a simple consequence of the more general equivalence
  \begin{align*}
    P'_N(B)\leq \hat P'(B^{\delta_1})+\delta_2~~ \forall B\subseteq \Xi'\iff\left\{
    \begin{array}{l}
      \exists T\in \mc P(\Xi'\times\Xi_N') ~\st~\Pi_{\Xi'} T=\hat P', ~\Pi_{\Xi'_N}T=P'_N,\\
      \textstyle\int_{\Xi'\times\Xi'_N}\one{\norm{\xi'-\xi'_i}\leq \delta_1}\d T(\xi', \xi_i')\geq 1-\delta_2
    \end{array}\right.
  \end{align*}
  for any $\delta_1, \delta_2\geq 0$ \cite[Theorem 11]{strassen1965existence}. 
\end{proof}

The equivalent formulation stated in Theorem \ref{thm:strassen-representation} can be solved efficiently using an off-the-shelf exponential cone optimization solver \citep{dahl2021primal} when both event sets $\Xi$ and $\Xi'$ have finite cardinality. Perhaps the only complication is that the size of this equivalent formulation counts $\mc O(\abs{\Xi'}\abs{\Xi'_N})$ variables for $\mc O(\abs{\Xi'}+\abs{\Xi'_N})$ constraints which may limit its practicality when both $\abs{\Xi'}$ and $\abs{\Xi_N'}$ are large.

We now indicate that even if the event set $\Xi'$ is not finite Problem (\ref{eq:strassen-representation}) still admits a finite reduction. Consider a finite partition of $\Xi'$ which is generated by the closed balls around the observed data points, i.e.,
\[
  \Xi'_k = {\displaystyle\cap_{i=1}^{\abs{\Xi'_N}}} \set{\xi'\in \Xi'}{\begin{array}{rl} \norm{\xi'-\xi'_i}\leq \delta & {\rm{if}~} b(k, i)=1\\\norm{\xi'-\xi'_i}> \delta & {\rm{if}~} b(k, i)=0\end{array}}
\]
for $0\leq k\leq K-1$ with $K = 2^{\abs{\Xi'_N}}$ and $b(k, i)$ the $i$th digit of the natural number $k$ written down in binary notation. Clearly, we have $\Xi'_k\cap \Xi'_{k'}=\emptyset$ if $k\neq k'$ and $\cup_{k=0}^{K-1} \Xi'_k=\Xi'$. Note that it may happen that $\Xi'_k=\emptyset$ for some $0\leq k\leq K-1$. We defer the proof of the following result to Appendix~\ref{sec:proof-lemma-ref}.

\begin{lemma}
  \label{lemma:strassen-representation-finite}
  Let $\Xi'_N={\rm{supp}}(P'_N)$ denote the support of $P'_N$ and let $\Xi$ be finite. Then,
  \begin{equation}
    \label{eq:strassen-representation-finite}
    \begin{array}{rl}
      \tilde c^\delta(z, P'_N) = \tilde c_f^\delta(z, P'_N) \defn \sup & \sum_{\xi\in \Xi} \ell(z, \xi) q_\xi\\[0.3em]
      \st & q_\xi\geq 0 ~\forall \xi\in \Xi, ~\hat p' \in \Re^K, ~t\in \Re^{K\times \abs{\Xi'_N}}\\[0.3em]
                                      & \sum_{k=0}^{K-1}\hat p'_k \log \left(\tfrac{\hat p'_k}{\sum_{\xi\in \Xi} O_\xi(\Xi'_k)\cdot q_\xi}\right)\leq r,  \\[0.3em]
                                      & \sum_{i=1}^{\abs{\Xi'_N}} t_{k,i}=\hat p'_k \quad 0\leq k\leq K-1,  \\[0.3em]
                                      & \sum_{k=0}^{K-1} t_{k,i}=P_N'(\xi'_i) \quad  1\leq i\leq \abs{\Xi'_N},\\[0.3em]
                                      & \sum_{i=1}^{\abs{\Xi'_N}} \sum_{k=0}^{K-1} b(k, i) t_{k,i}\geq 1-\delta.
    \end{array}
  \end{equation}
\end{lemma}

It should be remarked although Problem (\ref{eq:strassen-representation-finite}) is finite it may be very large as indeed we have that $K=\mc O(2^{\abs{\Xi'_N}})$. This clearly limits the practical usefulness of Problem (\ref{eq:strassen-representation-finite}) to situations where the number of observed distinct data points $\Xi'_N$ is very small. Surprisingly, using a dual equivalent formulation instead the computational burden of evaluating $c^\delta(z, P_N')$ can be further reduced to $\mc O(\abs{\Xi_N'})$ variables and $\mc O(1)$ constraints which is the subject of the following section.

\subsection{Dual Representation}

Consider the minimization problem
\begin{equation}
  \label{eq:dual-formulation}
  \begin{array}{r@{~~}l}

    {\displaystyle\inf} & \displaystyle \beta r\! -\!\int_{\Xi'} v(\xi') \, \d P'_N(\xi')  + \gamma (\delta-1) + \max_{\xi\in\Xi}\! \left[\ell(z, \xi) + \beta \!\!\int_{\Xi'}\!\!\! \exp\left(\frac{u(\xi')}{\beta}-1\right)\d O_\xi(\xi')\right]\\[0.3em]
    \st & \beta\geq 0, ~u:\Xi'\to\Re, ~v:\Xi'_N\to\Re,~\gamma\geq 0,\\[0.3em]
                        & \gamma \one{\norm{\xi'-\xi'_i}  \leq \delta}+ v(\xi'_i)\leq u(\xi')  \quad \forall \xi'\in \Xi', \,\xi'_i\in \Xi'_N.
  \end{array}
\end{equation}

We label the previous problem as the dual problem of the primal problem (\ref{eq:strassen-representation}) which is nontrivial if $\tilde c^\delta(z, P'_N)$ is finite. We remark that the next duality result does not assume $\Xi'$ to have finite cardinality. We defer its proof to Appendix \ref{sec:proof-theor-dual}.

\begin{theorem}[Dual Representation]
  \label{thm:dual-formulations}
  Let $\Xi$ be finite. Suppose that Problem \eqref{eq:naive-representation} enjoys the Slater constraint qualification condition, i.e., there exist $Q_s\in \mc P(\Xi)$, $\hat P'_s\in \mc P(\Xi')$ with $I(\hat P'_s, Q_s)<r$ and $\LP{\hat P'_s}{P'_N}<\delta$. Then, we have
  \(
  \tilde c^\delta(z, P'_N) = \eqref{eq:dual-formulation}.
  \)
\end{theorem}

First note that the dual characterization of the prediction function in Theorem \ref{thm:dual-formulations} allows for the convex saddle-point formulation \eqref{eq:dro-optimal} to be solved as an ordinary convex minimization problem as the dual formulation (\ref{eq:dual-formulation}) is jointly convex in $z\in Z$ and the dual variables. This may be desirable in practice as saddle-point optimization solvers are typically not as mature as solvers addressing standard optimization problems. This dual formulation counts $\mc O(\abs{\Xi_N'}+\abs{\Xi'})$ variables for $\mc O(\abs{\Xi'}\abs{\Xi'_N})$ constraints.

We now show that this can be further reduced to $\mc O(\abs{\Xi_N'})$ variables for $\mc O(1)$ constraints and in fact allows for a finite optimization characterization independent of the cardinality of the event set $\Xi'$. Consider $v:\Xi'_N\to\Re$ and define $J(v)\defn\abs{\set{v(\xi')}{\xi'\in \Xi'_N}}$. Let $\xi'_{[j]}(v)$ for $j\in [1,\dots, J(v)]$ denote any partition of the observed data points in $\Xi'_N$ such that the dual variables $v(\xi'_{[j]}(v))$ are non-increasing. That is, we have $\Xi'_N=[\xi'_{[1]}(v), \dots, \xi'_{[J(v)]}(v)]$ and
\( v(\tilde{\xi}')\geq v(\hat{\xi}') \) for all $\tilde{\xi}'\in \xi'_{[j]}(v), \hat{\xi}'\in \xi'_{[j']}(v)$ with $1\leq j<j'\leq J(v)$ as well as \( v(\tilde{\xi}')= v(\hat{\xi}') \) for all $\tilde{\xi}', \hat{\xi}'\in \xi'_{[j]}(v)$ with $1\leq j\leq J(v)$.
We may now partition the event set $\Xi'$ using the sets $\Xi'_0(v)=\emptyset$, $\Xi'_j(v) = \tset{\xi'\in \Xi'}{\min \tset{||{\xi'-\tilde{\xi}'}||}{\tilde{\xi}'\in \xi'_{[j]}(v)}\leq \delta} \setminus \cup_{l\in [1, \dots, j-1]} \Xi'_l(v)$ for $j\in [1, \dots, J(v)]$, and $\Xi'_{J(v)+1}(v)= \Xi' \setminus \cup_{l\in [1, \dots, J(v)]} \Xi'_l(v)$.

\begin{example}
  \label{ex:partition}
  Let $\Xi'_N = [\xi'_1, \dots, \xi'_5]$ and consider $v:\Xi'_N\to\Re$ given as $v(\xi')=1$ if $\xi'=\xi'_1$, $v(\xi')=3$ if $\xi'=\xi'_2$, $v(\xi')=4$ if $\xi'=\xi'_3$, $v(\xi')=2$ if $\xi'=\xi'_4$ and $v(\xi')=3$ if $\xi'=\xi'_5$.
  Then, $J(v)=4$ and $\xi'_{[1]}(v)=\{\xi'_3\}$, $\xi'_{[2]}(v)=\{\xi'_2,\xi'_5\}$, $\xi'_{[3]}(v)=\{\xi'_4\}$, $\xi'_{[4]}(v)=\{\xi'_1\}$. The associated partition of the event set $\Xi'=\cup_{j=1}^{J(v)+1} \Xi'_j(v)$ is given in Figure \ref{fig:partition}.
\end{example}

  \begin{figure}
    \centering
    \includegraphics[width=0.5\textwidth]{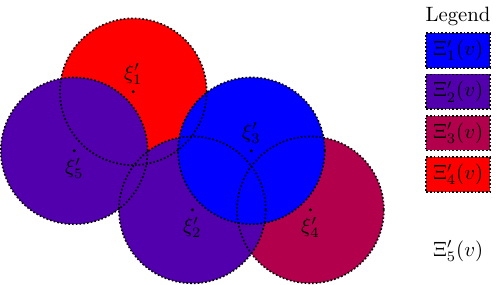}
    \caption{The associated partition of the event set $\Xi'_N=\cup_{j=1}^{J(v)+1} \Xi'_j(v)$ associated with the vector $v$ given in Example \ref{ex:partition}.}
    \label{fig:partition}
  \end{figure}

\begin{lemma}[Finite Convex Dual Representation]
  \label{lemma:dual-representation-finite}
  The dual problem admits the following convex formulation:
  \begin{equation}
    \label{eq:finite-dual-representation}
    \begin{array}{r@{~~}l}
      (\ref{eq:dual-formulation}) = {\displaystyle\inf} & \beta r\! -\!\sum_{j=1}^{J(v)} v(\xi'_{[j]}) P'_N(\xi'_{[j]})  + \gamma (\delta-1)  + \max_{\xi\in\Xi}~ \Big[ \ell(z, \xi) \\[0.75em]
                                                        & \hspace{1.7em}  + \sum_{j=1}^{J(v)} O_\xi(\Xi'_j(v))\frac{\beta}{e} \exp\left(\frac{\max(v(\xi'_{[1]}(v)), v(\xi'_{[j]}(v))+ \gamma)}{\beta}\right) + O_\xi(\Xi'_{J(v)+1}(v))\frac{\beta}{e} \exp\left(\frac{v(\xi'_{[1]}(v))}{\beta}\right)\Big]\\[0.75em]
      \st & \beta\geq 0,  ~v:\Xi'_N\to\Re,~\gamma\geq 0.
    \end{array}
  \end{equation}

\end{lemma}
\begin{proof}
Note that the optimal dual variable $u^\star(\gamma, v)$ in the dual formulation (\ref{eq:dual-formulation}) can be determined explicitly as a function of the dual variables $v$ and $\gamma$. 
  It is can be easily verified that optimal dual variable is of the form
  \[
    u^\star(\xi; \gamma, v) = \begin{cases}
      \max(v(\xi'_{[1]}(v)), v(\xi'_{[1]}(v))+ \gamma)& {\rm{if}} ~ \xi'\in \Xi'_1(v),\\
      \vdots & \vdots\\
      \max(v(\xi'_{[1]}(v)), v(\xi'_{[J]}(v))+ \gamma) & {\rm{if}} ~ \xi'\in \Xi'_{J(v)}(v),\\
      v(\xi'_{[1]}(v)) & {\rm{if}} ~ \xi'\in \Xi'_{J(v)+1}(v)\\
    \end{cases}
  \]
  and is hence a piecewise constant function on the partition $\{\Xi'_j(v)\}_{j=1}^{J(v)+1}$. Indeed, any other feasible solution $u$ in formulation (\ref{eq:dual-formulation}) is pointwise larger, i.e., $u^\star(\xi; \gamma, v) \leq u(\xi)$ for all $\xi\in \Xi'$. As the objective function in formulation (\ref{eq:dual-formulation}) is increasing (in the pointwise order) in $u$, the choice $u^\star$ is without loss of optimality. Consequently, the ultimate term in the objective of the dual formulation (\ref{eq:dual-formulation}) can be decomposed as the finite sum
  \begin{align*}
     \textstyle\beta \int_{\Xi'}&\textstyle\exp\left(\frac{u^\star(\xi', \gamma, v)}{\beta}-1\right)\d O_\xi(\xi') \\[0.5em]
   =& \textstyle \sum_{j=1}^{J(v)} O_\xi(\Xi'_j(v))\beta \exp\left(\tfrac{\max(v(\xi'_{[1]}(v)), v(\xi'_{[j]}(v))+ \gamma)}{\beta} -1\right) 
    \textstyle + O_\xi(\Xi'_{J(v)+1}(v))\beta \exp\left(\tfrac{v(\xi'_{[1]}(v))}{\beta}-1\right).
\end{align*}
The convexity of the reduced formulation is a consequence of the fact \citep[Theorem 1.31]{pennanen2019introduction} that partially minimizing a convex formulation (here formulation (\ref{eq:dual-formulation})) with respect to any variable (here the variable $u$) results in a reduced formulation (here formulation (\ref{eq:finite-dual-representation})) which is convex itself.%
\end{proof}

The dual representation (\ref{eq:finite-dual-representation}) can be solved as a finite convex optimization problem using a (stochastic) black-box optimization method \cite{nesterov1998introductory} as long as we have a (stochastic) oracle which can access the probabilities $O_\xi(\Xi'_j(v))=\int_{\Xi'} \one{\Xi'_j(v)} \, \d O_\xi(\xi') = \int_{\Xi'} \one{\Xi'_j(v)} \exp(-d(\xi, \xi'))\,\d m'(\xi')$ for all $\xi\in \Xi$ and $j \in [1, \dots, J(v)+1]$. The complexity of our efficient formulation \eqref{eq:naive-representation} is hence reduced to the complexity of integration of the noise distribution over certain intersections and unions of norm balls; see also Figure \ref{fig:partition}.

\section{Conclusions}

We introduce a novel class of distributionally robust optimization formulations which are efficient when making decisions based on noisy data in the sense pioneered by \cite{vanparys2020data}.
Their associated ambiguity sets are based on both an entropic divergence as well as an entropic optimal transport distance and can be interpreted as a large deviation rate function.
Unlike in the noiseless regime, we show that this rate function too irregular and a certain amount of smoothing is necessary.
Finally, we show that our efficient formulations are often tractable by exploiting a classical result by \citet{strassen1965existence} and considering a dual representation.

\bibliography{references}

\appendix
\clearpage

\section{Proofs}

\subsection{Proof of Theorem \ref{thm:ldp}}
\label{sec:proof-theorem-ldp}

\begin{proof}
  Let $T(P)$ be the joint distribution of the random variables $(\xi_i, \xi'_i)$ where Assumption \ref{ass:density} holds.
  We will show momentarily that
  \(
  \KL(\hat P', O\star P) = \min \tset{\KL(T', T(P))}{T'\in \mc P(\Xi\times\Xi'), ~\Pi_{\Xi'}T'=\hat P'}.
  \)
  The identity \eqref{eq:infimum_over_Q} follows then from Lemma \ref{sec:proof-theor-refthm:l}. The joint convexity of the rate function $I(\hat P', P)=\KL(\hat P', O\star P)$ follows from the joint convexity of the relative entropy function and the linearity of $O\star P$ in $P$.
  It remains to show that
  \(
  \KL(\hat P', O\star P) = \min \tset{\KL(T', T(P))}{T'\in \mc P(\Xi\times\Xi'),~\Pi_{\Xi'}T'=\hat P'}.
  \)

  Consider first the case where $\hat P' \not \ll O\star P$. In this case we have by definition $\KL(\hat P', O\star P)=\infty$.
  Furthermore, $\hat P'=\Pi_{\Xi'}T' \not \ll \Pi_{\Xi'}T(P) = O\star P\implies T'\not \ll T(P)$ and consequently $\min \tset{\KL(T', T(P))}{T'\in \mc P(\Xi'\times\Xi),~\Pi_{\Xi'}T'=\hat P'}=\infty$ as well.
  
  Consider now the case where $\hat P' \ll O\star P$. We will show (i) that here \(
  \KL(\hat P', O\star P) = \min \tset{\KL(T', T(P))}{T'\in \mc P(\Xi'\times\Xi),~\Pi_{\Xi'}T'=\hat P'} = \KL(T'^\star, T(P))
  \)
  and (ii) the minimum is achieved at $T'^\star\ll T(P)$ explicitly characterized as
  \begin{equation}
    \label{eq:explicit-Tstar}
    \frac{\d T'^\star}{\d T(P)}(\xi, \xi')= \frac{\d \hat P'}{\d (O\star P)}(\xi') \quad \forall \xi\in \Xi, ~\xi'\in \Xi'.
  \end{equation}

  First, observe that we have here the inequalities
  \begin{align}
     & \KL(\hat P', O\star P) \nonumber\\
    = & \int_{\Xi'} \log\left( \frac{\d \hat P'}{\d (O\star P)}(\xi')\right) \frac{\d \hat P'}{\d (O\star P)}(\xi') \, \d (O\star P)(\xi')\nonumber\\
    = & \int_{\Xi'}  \left[\max_{v\in \Re} v \frac{\d \hat P'}{\d (O\star P)}(\xi') - \exp(v-1)\right] \d (O\star P)(\xi')\nonumber\\
    = & \max_{v:\Xi'\to\Re} \int_{\Xi'} v(\xi') \d \hat P'(\xi') - \int_{\Xi'} \exp(v(\xi')-1) \d (O\star P)(\xi')\nonumber\\
    = & \max_{v:\Xi'\to\Re} \int_{\Xi'} v(\xi') \d \hat P'(\xi') - \int_{\Xi\times \Xi'}\exp(v(\xi')-1)\d T(P)(\xi, \xi')\nonumber\\
    = & \max_{v:\Xi'\to\Re} \int_{\Xi'} v(\xi') \d \hat P'(\xi') - \int_{\Xi\times \Xi'}\max_{R\geq 0} R v(\xi') - R \log(R) \d T(P)(\xi, \xi') \nonumber\\
    = & \max_{v:\Xi'\to\Re} \min_{R:\Xi\times \Xi'\to\Re_+} \int_{\Xi'} v(\xi') \d \hat P'(\xi') - \int_{\Xi\times \Xi'} R(\xi, \xi') v(\xi') - R(\xi, \xi') \log(R(\xi, \xi')) \d T(P)(\xi, \xi')\nonumber\\
    \leq & \max_{v:\Xi'\to\Re} \min_{T' \ll T(P)\in \mc M_+(\Xi\times \Xi')} \int_{\Xi'} v(\xi') \d \hat P'(\xi') - \int_{\Xi\times \Xi'} v(\xi') \d T'(\xi, \xi') + \log\left(\frac{\d T'}{\d T(P)}(\xi, \xi')\right)  \d T'(\xi, \xi')\nonumber\\
    \leq & \min_{T' \ll T(P)\in \mc M_+(\Xi\times \Xi')} \max_{v:\Xi'\to\Re}  \int_{\Xi'} v(\xi') \d \hat P'(\xi') - \int_{\Xi\times \Xi'} v(\xi') \d T'(\xi, \xi') + \log\left(\frac{\d T'}{\d T(P)}(\xi, \xi')\right)  \d T'(\xi, \xi')\nonumber\\
    = & \min_{T' \ll T(P)\in \mc M_+(\Xi\times \Xi'), ~\Pi_{\Xi'}T'=\hat P'} \log\left(\frac{\d T'}{\d T(P)}(\xi, \xi')\right)  \d T'(\xi, \xi') \nonumber\\
  \label{ineq:KL-divergence}  = &\min \tset{\KL(T', T(P))}{T'\in \mc P(\Xi\times\Xi'), ~\Pi_{\Xi'}T'=\hat P'}.
  \end{align}
  The first equality follows from the fact that $\hat P'\ll O\star P$ and the definition of the Kullback-Leibler divergence where we remark that here $\int \d \hat P'(\xi') = \int \d (O\star P)(\xi') = 1$.
  The second equality follows from the convex conjugate identity $\max_{v} a v - \exp(v-1) = a v^\star(a) - \exp(v^\star(a)-1) = a \log\left(a\right)$ with $v^\star(a) = 1 + \log(a)$ for all $a\in \Re_+$.
  The third equality follows from \citet[Theorem 14.60]{rockafellar2009variational}.
  The fourth equality follows from the fact that $\Pi_{\Xi'}T(P) = O\star P$. The fifth equality follows from the convex conjugate identity $\max_{R\geq 0} a R - R \log\left(R\right) = a R^\star(a) - R^\star(a) \log(R^\star(a))= \exp(a-1)$  with $R^\star(a)=\exp\left(a-1\right)$ for all $a\in\Re$.
  The sixth equality follows again from \citet[Theorem 14.60]{rockafellar2009variational}.
  The first inequality follows from the fact that with any $T'\ll T(P) \in \mc M_+(\Xi\times \Xi')$ we can associate the Radon-Nikodym derivative $R(\xi, \xi') = \tfrac{\d T'}{\d T(P)}(\xi, \xi')$. The second inequality follows from the standard minimax inequality.
  The penultimate equality follows from the fact that 
  \(
  \max_{v:\Xi'\to\Re}  \int_{\Xi'} v(\xi') \d \hat P'(\xi') - \int_{\Xi\times \Xi'} v(\xi') \d T'(\xi, \xi') = \chi_\infty(\Pi_{\Xi'}T'=\hat P').
  \)
  The ultimate equality follows from the fact that $T' \ll T(P)\in \mc M_+(\Xi\times \Xi'), ~\Pi_{\Xi'}T'=\hat P', ~\int \d \hat P'(\xi')=1\implies T' \ll T(P)\in \mc P(\Xi\times \Xi')$.

  Recall the definition of $T'^\star$ stated in Equation \eqref{eq:explicit-Tstar}.
  We have the chain of equalities
  \begin{align*}
   \int_{\Xi\times \Xi'} \one{\xi'\in B} \d T'^\star(\xi, \xi') = & \int_{\Xi\times\Xi'} \one{\xi'\in B} \tfrac{\d \hat P'}{\d (O\star P)}(\xi') \, \d T(P)(\xi, \xi')\\
    =&  \int_{\Xi'} \one{\xi'\in B} \tfrac{\d \hat P'}{\d (O\star P)}(\xi') \d (O\star P)(\xi') = \hat P'(B)
  \end{align*}
  for any measurable set $B\subseteq \Xi'$. The first equality follows from Equation \eqref{eq:explicit-Tstar} and the second equality exploits $\Pi_{\Xi'} T(P) = O\star P$. In particular this implies that $T^\star\in \mc P(\Xi\times \Xi')$ when choosing $B=\Xi'$ and $\Pi_{\Xi'}T^\star=\hat P$.
  As we have here $\hat P'\ll O\star P$ we can define the functions
  \(
    v^\star(\xi') = \log (\tfrac{\d \hat P'}{\d (O\star P)}(\xi'))+1
  \)
  and
  \(
    R^\star(\xi, \xi')\defn \tfrac{\d T'^\star}{\d T(P)}(\xi, \xi') = \tfrac{\d \hat P'}{\d (O\star P)}(\xi') = \exp\left(v^\star(\xi') -1\right) 
  \)
  for all  $\xi\in \Xi$  and $\xi'\in \Xi'$.
  Observe now that
  \begin{align*}
    & \KL(T'^\star, T(P))\\
      = &  \int_{\Xi\times \Xi'} \log\left(\frac{\d T'^\star}{\d T(P)}(\xi, \xi')\right)  \d T'^\star(\xi, \xi')\\
    = & \int_{\Xi'} v^\star(\xi') \d \hat P'(\xi') - \int_{\Xi\times \Xi'} v^\star(\xi') + \log\left(\frac{\d T'^\star}{\d T(P)}(\xi, \xi')\right) \, \d T'^\star(\xi, \xi')\\
    = & \int_{\Xi'} v^\star(\xi') \d \hat P'(\xi') - \int_{\Xi\times \Xi'} v^\star(\xi') \frac{\d T'^\star}{\d T(P)}(\xi, \xi')  + \log\left(\frac{\d T'^\star}{\d T(P)}(\xi, \xi')\right)  \frac{\d T'^\star}{\d T(P)}(\xi, \xi') \, \d T(P)(\xi, \xi')\\
    = & \int_{\Xi'} v^\star(\xi') \d \hat P'(\xi') - \int_{\Xi\times \Xi'} R^\star(\xi, \xi') v^\star(\xi') - R^\star(\xi, \xi') \log(R^\star(\xi, \xi'))\, \d T(P)(\xi, \xi')\\
    = & \int_{\Xi'} v^\star(\xi') \d \hat P'(\xi') - \int_{\Xi\times \Xi'}\exp(v^\star(\xi')-1)\, \d T(P)(\xi, \xi')\\
    = &\int_{\Xi'} v(\xi') \d \hat P'(\xi') - \int_{\Xi'}\exp(v(\xi')-1)\, \d (O\star P)(\xi')\\
    = & \int_{\Xi'} v(\xi') \d \hat P'(\xi') + \int_{\Xi'} \d \hat P(\xi') - \int_{\Xi'}\d (O\star P)(\xi') = \KL(\hat P', O\star P)
  \end{align*}
  The first equality follows from the definition of the Kullback-Leibler divergence and the fact that both $T'^\star$ and $T(P)$ are in $\mc P(\Xi\times \Xi')$.  The second equality follows from the fact that $\int_{\Xi'} v^\star(\xi') \d \hat P'(\xi') - \int_{\Xi\times \Xi'} v^\star(\xi') \d T'^\star(\xi, \xi') = 0$ as $\Pi_{\Xi'}T'^\star=\hat P'$ and the linearity of integration. The third, fourth and fifth equality follow from the fact that $T^\star\ll T(P)$ with Radon-Nikodym derivative $R^\star(\xi, \xi')\defn \tfrac{\d T^\star}{\d T(P)}=\exp(v^\star(\xi')-1)$.
  The penultimate equality follows from $\Pi_{\Xi'}T(P) = O\star P$.
  The final equality is direct consequence of the identity $v^\star(\xi') = \log (\tfrac{\d \hat P'}{\d (O\star P)}(\xi'))+1$.
  As we have already shown that $\KL(\hat P', O\star P)\leq \min \tset{\KL(T', T(P))}{T'\in \mc P(\Xi\times\Xi'), ~\Pi_{\Xi'}T'=\hat P'}$ in Equality \eqref{ineq:KL-divergence} this implies that as $T'^{\star}\in \mc P(\Xi\times\Xi'), ~\Pi_{\Xi'}T'^\star=\hat P'$ we have $\KL(\hat P', O\star P)= \min \tset{\KL(T', T(P))}{T'\in \mc P(\Xi\times\Xi'), ~\Pi_{\Xi'}T'=\hat P'}$.

  We remark that in the case where $\hat P' \ll O\star P$ we have
  \begin{align*}
    \Pi_\Xi T'^\star (B) = &\int_{\Xi\times \Xi'} \one{\xi\in B} \d T'^\star(\xi, \xi')\\
    = &  \int_{\Xi\times \Xi'}  \one{\xi\in B} \frac{\d \hat P'}{\d (O\star P)}(\xi') \d T(P)(\xi, \xi')\\
    = & \int_{\Xi\times \Xi'}  \one{\xi\in B} \left(\frac{\d \hat P'}{\d (O\star P)}(\xi') \exp(-d(\xi,\xi'))  \d m'(\xi')\right) \d P(\xi)\\
    = & \int_{\Xi}  \one{\xi\in B} \left(\int_{\Xi'}\frac{\exp(-d(\xi,\xi'))}{\int_{\Xi} \exp(-d(\xi'', \xi'))\d P(\xi'')} \d \hat P'(\xi') \right) \d P(\xi)
  \end{align*}
  for all $B\subseteq \Xi$. Hence,
  \(
    \tfrac{\Pi_\Xi T'^\star}{\d P}(\xi) = \int_{\Xi'}\tfrac{\exp(-d(\xi,\xi'))}{(\int_{\Xi} \exp(-d(\xi'', \xi'))\d P(\xi''))} \d\hat P'(\xi') 
    \)
   for all $\xi\in \Xi$.
\end{proof}

\begin{lemma}
  \label{sec:proof-theor-refthm:l}
  We have under Assumption \ref{ass:density} that
  \begin{align*}
     \inf & \set{\KL(T', T(P))}{T'\in \mc P(\Xi\times\Xi'),~\Pi_{\Xi'}T'=\hat P'}
    = \inf \set{W_1(\hat P', Q) + \KL(Q, P)+ \KL(\hat P', m')}{{Q\in \mc P(\Xi)}}.
  \end{align*}
\end{lemma}
\begin{proof}
  Under Assumption \ref{ass:density} we have $T(P)\ll P\otimes m'$ and $P\otimes m'\ll T(P)$ with
  \(
  \tfrac{\d T(P)}{\d (P\otimes m')}(\xi, \xi') = \exp(-d(\xi,\xi'))
  \)
  and
  \(
  \tfrac{\d (P\otimes m')}{\d T(P)}(\xi, \xi') = \exp(d(\xi,\xi')).
  \)
  Assume that $T'\ll T(P)$ as otherwise $\KL(T', T(P))=+\infty$. The assumption $T'\ll T(P)$ furthermore implies $\Pi_\Xi T'\ll \Pi_\Xi T(P)=P$ and $\Pi_{\Xi'}T'\ll \Pi_{\Xi'} T(P) = P' \ll  m'$. We have here hence
  \begin{align}
    & {\KL}(T', T(P)) \nonumber\\[0.5em]
    = & \int_{\Xi\times \Xi'} \log\left(\frac{\d T'}{\d T(P)}(\xi, \xi')\right) \d T'(\xi, \xi') \nonumber\\[0.5em]
    = & \int_{\Xi\times \Xi'} \log \left(\frac{\d T'}{\d P\otimes m'}(\xi, \xi') ~ \frac{\d P\otimes m'}{\d T(P)}(\xi, \xi')  \right) \,\d T'(\xi, \xi') \nonumber\\[0.5em]
    = & \int_{\Xi\times \Xi'} \log \left(\frac{\d T'}{\d P\otimes m'}(\xi, \xi')\right) \,\d T'(\xi, \xi')  - \int_{\Xi\times \Xi'} \log \left(\frac{\d T(P)}{\d P\otimes m'}(\xi, \xi')\right) \,\d T'(\xi, \xi')\nonumber\\[0.5em]
    = & {\KL}(T', P\otimes m') - \int_{\Xi\times \Xi'} \log \left(\exp(-d(\xi, \xi'))\right) \,\d T'(\xi, \xi')\nonumber\\[0.5em]
  \label{eq:KL-interpretation}  = & \int_{\Xi\times \Xi'} d(\xi, \xi')\,\d T'(\xi, \xi') + {\KL}(T', \Pi_\Xi T'\otimes \Pi_{\Xi'}T')+{\KL}(\Pi_\Xi T', P) + \KL(\Pi_{\Xi'}T', m')\geq 0.
  \end{align}
  The first equality follows from the definition of the Kullback-Leibler divergence and the fact that $T'$ and $T(P)$ are in $\mc P(\Xi\times \Xi')$. The second equality follows from the fact that $T'\ll P\otimes m'$ and $P\otimes m' \ll T(P)$. The third and fourth equalities follow from 
  $T(P)\ll P\otimes m'$ and $P\otimes m'\ll T(P)$ with
  \(
  \tfrac{\d T(P)}{\d (P\otimes m')}(\xi, \xi') = \exp(-d(\xi,\xi')).
  \)
  The final equality follows from \cite[Lemma 6.1]{rigollet2018entropic}. Hence, we have
  \begingroup
  \allowdisplaybreaks
  \begin{align*}
    & \inf \set{\KL(T', T(P))}{T'\in \mc P(\Xi\times\Xi'),~ \Pi_{\Xi'} T'=\hat P'}\\[0.5em]
    =& \inf \set{\KL(T', T(P))}{Q\in \mc P(\Xi),~T'\ll T(P)\in\mc P(\Xi\times\Xi') ,~ \Pi_{\Xi} T'=Q, ~\Pi_{\Xi'}T'=\hat P'}\\[0.5em]
    =& \inf \, \big\{\textstyle\int d(\xi, \xi')\,\d T'(\xi, \xi') + \KL(T', \Pi_\Xi T'\otimes \Pi_{\Xi'}T')+\KL(Q, P) + \KL(\hat P', m') \\
    & \hspace{20em}: Q\in \mc P(\Xi), ~T'\ll T(P)\in\mc P(\Xi\times\Xi'), ~T'\in \mc T(\hat P', Q)\big\}\\[0.5em]
    =& \inf \, \big\{\textstyle\int d(\xi, \xi')\,\d T'(\xi, \xi') + \KL(T', \Pi_\Xi T'\otimes \Pi_{\Xi'}T')+{\KL}(Q, P) + \KL(\hat P', m') \\
    & \hspace{20em}: Q\in \mc P(\Xi), ~T'\in\mc P(\Xi\times\Xi'), ~T'\in \mc T(\hat P', Q)\big\}\\[0.5em]
    = & \inf_{Q\in \mc P(\Xi)} W_1(\hat P', Q) + \KL(Q, P)+ \KL(\hat P', m')
  \end{align*}
  \endgroup  
  establishing the claim. The first equality follows here from the fact that $T'\not \ll T(P) \implies \KL(T', T(P))=\infty$. The second equality follows from Equation \eqref{eq:KL-interpretation}.
  Recall that we may assume that $Q\ll P$ and $P'\ll m'$ as otherwise $\KL(Q, P)=\infty$ or $\KL(\hat P', m')=\infty$, respectively.
  Hence, we have $T' \ll \Pi_\Xi T'\otimes \Pi_{\Xi'}T' = Q \otimes \hat P' \ll P\otimes m' \ll T(P)$ from which the third equality follows immediately.
  The final equality follows from the definition of entropic optimal transport in Equation \eqref{def:entropic-wasserstein}.
\end{proof}

{\color{black}
\subsection{Proof of Proposition \ref{prop:confidence}}

Let $C(\delta, A)$ denote the minimal covering number of a set $A$ in a metric space $S$ using metric norm balls of radius $\delta$.
\citet[Lemma 1]{kulkarni1995general} derive upper and lower bounds on covering numbers $C(\delta, \mc P(\Xi'))$ assuming the event set $\Xi'$ is bounded.
We derive here first a similar bound where bounded variance substitutes for bounded support. 

\begin{lemma}
  \label{lemma:cover-bound}
  Let $\mc P_V(\Xi')$ be the simplex of probability distributions $\mu$ on $\Xi'$ with $\int \norm{\xi'}^2 \d \mu(\xi')< V$. Then,
  \[
    C(\delta, \mc P_V(\Xi'))\leq 2\left(\tfrac{2e}{\delta}\right)^{C(\delta, \Xi'_\delta)}
  \]
  where $\Xi'_\delta \defn \Xi' \cap \{\norm{\xi'}^2\leq \tfrac{2V}{\delta}\}$.
\end{lemma}
\begin{proof}

  We set out along the same lines as the proof of \cite[Lemma 1]{kulkarni1995general} and attempt to prove the lemma by explicitly constructing an $\epsilon$-cover of $\mc P_V(\Xi')$.

  Let $\tilde \xi'_1, \dots, \tilde \xi'_{C(\delta, \Xi'_\delta)}$ be the centers of a set of $\delta$-norm balls which cover $\Xi'_\delta$. We shall consider only distributions in our cover which are supported on this finite collection of points in $\Xi'$.
  Define indeed the measures
  \[
    \nu_i^j =\frac{j \delta}{2C( \delta, \Xi'_\delta)}\cdot \delta_{\xi'_i} \quad \forall i=1, \dots, C( \delta, \Xi'_\delta), ~\forall j=0, 1, \dots, \frac{2C( \delta, \Xi'_\delta)}{\delta}
  \]
  and the set
  \[
    \mc S = \set{\nu\in\mc P(\Xi')}{\exists (i_1, j_1), \dots, (i_k, j_k) ~\st~ \mu = \textstyle\sum_{\alpha=1}^k \nu_{i_\alpha}^{j_\alpha}}.
  \]
  Note that the set $\mc S$ is finite as indeed it counts at most $(\tfrac{2C( \delta, \Xi'_\delta)}{\delta}+1)^{C( \delta, \Xi'_\delta)}$ distributions.
  Also, note that $\mc S$ is an $\delta$-cover of $\mc P_V(\Xi')$, that is, for any $\mu \in \mc P_V(\Xi')$ there exists an $\nu \in \mc S$ such that $\LP{\mu}{\nu}\leq \delta$. To see that, choose $\nu$ the following approximation to $\mu$. Let $i_\alpha=\alpha$, $\alpha=1,\dots, 2C( \delta, \Xi'_\delta)-1$ and choose
  \[
    j_\alpha = \lfloor \mu \left[ \left(\Xi'_\delta\cap B(\xi'_{i_\alpha}, \delta)\right)\setminus \left(\cup_{k=1}^{\alpha-1}B(\xi'_{i_k}, \delta)\right) \right] \tfrac{2C( \delta, \Xi'_\delta)}{\delta} \rfloor
  \]
  and finally
  \[
    j_{C( \delta, \Xi'_\delta)} = \frac{2C( \delta, \Xi'_\delta)}{\delta}-\sum_{\alpha=1}^{C( \delta, \Xi'_\delta)-1} j_\alpha.
  \]
  Here we take $B(\xi', \delta)$ to mean the closed norm ball in $\Xi'$ centered at $\xi'$ of radius $\delta$.
  Now take $\nu = \sum_{\alpha=1}^{C( \delta, \Xi'_\delta)} \nu_{i_\alpha}^{j_\alpha}\in \mc S$. By construction of the final $j_{C( \delta, \Xi'_\delta)}$ we have that $\nu$ must be a probability distribution. 
  Furthermore, for any measurable set  $E\subseteq \Xi'$ we have
  \begingroup
  \allowdisplaybreaks
  \begin{align*}
    \mu[E]= & \mu[E\setminus \Xi'_\delta]+\mu [E\cap \Xi'_\delta]\\
    \leq & \delta/2+ \mu [E\cap \Xi'_\delta]\\
    \leq & \delta/2+ \sum_{\alpha=1}^{C( \delta, \Xi'_\delta)} \mu\left[ (E\cap \Xi'_\delta) \cap \left( B(\xi'_{i_\alpha}, \delta) \setminus \left(\cup_{k=1}^{\alpha-1}B(\xi'_{i_k}, \delta)\right) \right) \right]\\
    \leq & \delta/2+ \sum_{\alpha=1}^{C( \delta, \Xi'_\delta)} \mu\left[ B(\xi'_{i_\alpha}, \delta) \setminus \left(\cup_{k=1}^{\alpha-1}B(\xi'_{i_k}, \delta)\right)  \right] \one{\xi'_{i_\alpha}\in E^\delta}\\
    \leq & \delta/2+\sum_{\alpha=1}^{C( \delta, \Xi'_\delta)} \left( j_\alpha \frac{\delta}{2C( \delta, \Xi'_\delta)} + \frac{\delta}{2C( \delta, \Xi'_\delta)} \right) \one{\xi'_{i_\alpha}\in E^\delta}\\
    \leq & \delta/2+\sum_{\alpha=1}^{C( \delta, \Xi'_\delta)} \left( j_\alpha \frac{\delta}{2C( \delta, \Xi'_\delta)}  \right) \one{\xi'_{i_\alpha}\in E^\delta} +\delta/2\\
    =& \nu[E^\delta]+\delta.
  \end{align*}
  \endgroup
  The first inequality follows from Chebyshev inequality $\mu(\Xi'\setminus \Xi'_\delta) \leq \mu(\norm{\xi'}>\sqrt{\tfrac{2V}{\delta}})\leq \delta/2$ as $\mu\in \mc P_V(\Xi')$.
  The second inequality follows from the fact that $\cup_{\alpha=1}^{C(\delta, \Xi'_\delta )}B(\xi_\alpha', \delta) \supseteq \Xi'_\delta$.
  The third inequality follows from the observation that if $\xi'_{i_\alpha}\not\in E^\delta$ then clearly $B(\xi'_{i_\alpha}, \delta)\cap E=\emptyset$.
  The fourth inequality is justified by
  \[
    j_\alpha+1 \geq \mu \left[ \left(\Xi'_\delta\cap B(\xi'_{i_\alpha}, \delta)\right)\setminus \left(\cup_{k=1}^{\alpha-1}B(\xi'_{i_k}, \delta)\right) \right] \tfrac{2C( \delta, \Xi'_\delta)}{\delta}.
  \]
  The fifth inequality then follows from $\sum_{\alpha=1}^{C( \delta, \Xi'_\delta)}\tfrac{\delta}{(2C( \delta, \Xi'_\delta))}\one{\xi'_{i_\alpha}\in E^\delta}\leq C( \delta, \Xi'_\delta) \tfrac{\delta}{(2C( \delta, \Xi'_\delta))} \leq \delta/2$. The final inequality follows from the observation that $\nu[E^\delta]=\sum_{\alpha=1}^{C( \delta, \Xi'_\delta)} \left( j_\alpha \tfrac{\delta}{(2C( \delta, \Xi'_\delta))}  \right) \one{\xi'_{i_\alpha}\in E^\delta}$.
  As $\mu[E]\leq \nu[E^\delta]+\delta$ holds for any measurable set $E$ we have by definition of the L\'evy-Prokhorov metric that $\LP{\nu}{\mu}\leq \delta$.

  Now straightforwardly adapting the cardinality estimate provided in the proof of \cite[Lemma 1]{kulkarni1995general} we get $\abs{\mc S} \leq 2 \left(\tfrac{2e}{\delta}\right)^{C(\delta, \Xi'_\delta)}$ and the result follows.
\end{proof}

\begin{lemma}
  \label{lemma:wlln-2}
  Let $\int \norm{\xi'}^2 \d P'(\xi')\leq V'$ and $\int \norm{\xi'}^4 \d P'(\xi')\leq K'^2$. Then, we have that
  \[
    \lim_{N\to\infty} \Prob[P_N'\not\in \mc P_{V'+\Delta}(\Xi')] \leq \tfrac{K'^2}{(\Delta^2 N)}.
  \]
  for any $\Delta>0$.
\end{lemma}
\begin{proof}

  Denote with $F'$ the distribution of the random variable $R=\norm{\xi'}^2$. We have 
  $\E{F'}{R} \leq V'$ and observe $\E{F'}{(R-\E{F'}{R})^2}\leq \E{F'}{R^2}=K;^2\defn\E{P'}{\norm{\xi'}^4}$. Let $F'_N = \frac 1N \sum_{i=1}^N \delta_{\norm{\xi'_i}^2}$. From Chebyshev's inequality it follows that
  \begin{align*}
    \lim_{N\to\infty}\Prob[\textstyle\int \norm{\xi'}^2 \d P'_N(\xi')> V'+\Delta]  \leq & \lim_{N\to\infty} \Prob[\textstyle\int r \d F'_N(r) > 2V']\\   
    \leq & \lim_{N\to\infty} \Prob[\textstyle\int r \d F'_N(r) - \int r \d F'(r)  > \Delta] \\
    \leq & \lim_{N\to\infty} \Prob[\abs{\textstyle\int r \d F'_N(r) - \int r \d F'(r)}  > \frac{K'}{\sqrt{N}} \frac{\Delta\sqrt{N}}{K'}] \\
    \leq & \,\tfrac{K'^2}{(\Delta^2 N)}
  \end{align*}
  establishing the claim.
\end{proof}

\begin{proof}[Proof of Proposition \ref{prop:confidence}]
  Let $V'>0$ be an upper bound on the variance  and $K'^2$ an upper bound on the fourth order moment of $P'$, i.e., $\int \norm{\xi'}^2 \d P'(\xi') < V'$ and $\int \norm{\xi'}^2 \d P'(\xi') \leq \int \norm{\xi'}^4 \d P'(\xi')\leq K'^2<\infty$. Let $\mc P_{V'+\Delta}(\Xi')$ be the set of distributions on $\Xi'$ with variance bounded by $V'+\Delta$ for some arbitrary $\Delta>0$. Define a set of bad events $E_N=\tset{\hat P'\in \mc P(\Xi')}{\exists! Q ~\st~ \LP{\hat P'}{Q}\leq \delta_N,~\KL(Q, P')\leq r_N}$. We have
  \begin{align}
    \Prob\left[\exists Q ~\st~ \LP{P_N'}{Q}\leq \delta_N,~\KL(Q, P')\leq r_N\right]  =& 1-\Prob\left[P'_N\in E_N\right], \nonumber\\
    \Prob\left[P'_N\in E_N\right] \leq & \Prob\left[P'_N\not\in \mc P_{V'+\Delta}(\Xi') \right] + \Prob\left[P'_N\in E_N\cap \mc P_{V'+\Delta}(\Xi')\right]. \label{eq:overall-bound}
  \end{align}
  In Lemma \ref{lemma:wlln-2} we show that
  \begin{equation}
    \label{eq:non-exponential-finite-sample-part}
    \Prob\left[P'_N\not\in \mc P_{V'+\Delta}(\Xi') \right]\leq \tfrac{K'^2}{(\Delta^2 N)}
  \end{equation}
  as a direct consequence of Chebyshev's inequality and hence also $\lim_{N\to\infty}\Prob\left[P'_N\not\in \mc P_{V'+\Delta}(\Xi') \right]=0$.
  It remains to show that also $\lim_{N\to\infty} \Prob\left[P'_N\in E_N\cap \mc P_{V'+\Delta}(\Xi')\right]=0$. According to \cite[Equation 4.5.7 \& Lemma 6.2.13]{dembo2009large} we have the finite sample bound
  \begin{align*}
    \Prob\left[P'_N\in E_N\cap \mc P_{V'+\Delta}(\Xi')\right] & \leq  C(\delta_N, E_N\cap \mc P_{V'+\Delta}(\Xi')) \exp(-N \inf_{\LP{Q}{Q'}\leq \delta_N, Q\in E_N} \KL(Q', P'))\\
                                                          & \leq  C(\delta_N, \mc P_{V'+\Delta}(\Xi')) \exp(-N r_N).
  \end{align*}
  The second equality follows from $Q\in E_N \implies KL(Q', P')>r_N$ for any $\LP{Q}{Q'}\leq \delta_N$. Furthermore, in Lemma \ref{lemma:cover-bound} we show that
  \(
  C(\delta_N, \mc P_{V'+\Delta}(\Xi')) \leq 2\left(\tfrac{2e}{\delta_N}\right)^{C(\delta_N, \Xi'_{\delta_N})}
  \)
  where $\Xi'_{\delta_N} \defn \Xi' \cap \{\norm{\xi'}^2\leq \tfrac{2(V'+\Delta)}{\delta_N}\}$. Finally, we have the bound $C(\delta_N, \Xi'_{\delta_N}) \leq (4 (V'+\Delta) \sqrt{\dim(\Xi')}/\delta_N^2)^{\dim(\Xi')}$ by \citet{shalev2014understanding}. Hence,
  \begin{align}
    \log \Prob\left[P'_N\in E_N\cap \mc P_{V'+\Delta}(\Xi')\right]&\leq \log C(\delta_N, \mc P_{V'+\Delta}(\Xi')) - N r_N\nonumber\\
    & \leq \log 2 + C(\delta_N, \Xi'_{\delta_N}) \log\left(\tfrac{2e}{\delta_N}\right) - N r_N\nonumber\\
    & \leq \log 2 + (4 (V'+\Delta) \sqrt{{\dim(\Xi')}}/\delta_N^2)^{\dim(\Xi')} \log\left(\tfrac{2e}{\delta_N}\right) - N r_N\nonumber\\
    &\leq \log 2 + (4 (V'+\Delta) \sqrt{{\dim(\Xi')}})^{\dim(\Xi')}\log\left(\tfrac{2e}{\delta_N}\right)\delta^{-2{\dim(\Xi')}}_N - N r_N. \label{eq:exponential-finite-sample-part}
  \end{align}
    As for $r_N = \Omega(N^{\gamma-1})$ and $\delta_N = \Omega(N^{-\tfrac{\gamma'}{(2{\dim(\Xi')})}})$ the left hand side of the final inequality is unbounded from below as $N$ tends to infinity we have shown that $\lim_{N\to\infty} \Prob\left[P'_N\in E_N\cap \mc P_{V'+\Delta}(\Xi')\right]=0$ from which the claim follows.
    We remark that by simply combining \eqref{eq:overall-bound}, \eqref{eq:non-exponential-finite-sample-part} and \eqref{eq:exponential-finite-sample-part} we can also derive the finite sample guarantee
  \begin{align}
    & \Prob\left[\exists Q ~\st~ \LP{P_N'}{Q}\leq \delta_N,~\KL(Q, P')\leq r_N\right]\nonumber\\
    \geq & 1-[\textstyle\inf_{\Delta>0} \tfrac{K'^2}{(\Delta N)} + 2 (2e/\delta_N)^{({4(V'+\Delta)\sqrt{\dim(\Xi')}\delta_N^{-2})}^{\dim(\Xi')}}\exp(-r_N N)]     \label{eq:finite-sample-bound}
  \end{align}
  as here the parameter $\Delta>0$ was arbitrary.
\end{proof}

\subsection{Proof of Theorem \ref{thm:consistency}}

As remarked in Section \ref{sec:consistency} consistency is only possible when it holds that $P$ can be identified from $P'$, i.e., $\set{Q\in \mc P}{O\star Q=O\star P}=\{P\}$. In the additive noise setting introduced in Example \ref{ex:additive-error} the event sets are $\Xi$ and $\Xi'=\Re^{\dim(\Xi)}$.
The next result establishes a slightly stronger identifiability result than the one discussed using L\'evy's continuity theorem.

\begin{lemma}
  \label{lemma:weak-convergence}
  Let $O^{AE}$ be the additive error observational model discussed in Example \ref{ex:gaussian-noise} with a noise distribution $E$ satisfying Assumption \ref{ass:identifiability}.
  Consider any sequence $\delta'_N\geq 0$ with $\lim_{N\to\infty}\delta'_N=0$. Then, there exists a nonincreasing sequence $\delta''_N\geq 0$ with $\lim_{N\to\infty}\delta''_N=0$ so that
  \[
    \set{Q\in \mc P}{\LP{O^{AE}\star Q}{O^{AE}\star P}\leq \delta'_N}\subseteq \set{Q\in \mc P}{\LP{Q}{P}\leq \delta''_N}.
  \]
\end{lemma}
\begin{proof}
  For the sake of contradiction, assume the statement is false and hence there must exist a sequence $\{Q_N\}_{N\geq 1}$ so that $\LP{O^{AE}\star Q_N}{O^{AE}\star P}\leq \delta_N'$ and $\lim_{N\to\infty}\LP{Q_N}{P}>0$. That is, we can find a sequence $\{Q_N\}_{N\geq 1}$ so that $O^{AE}\star Q_N$ converges weakly to $O^{AE}\star P$ but $Q_N$ does not converge weakly to $P$. As shown below this leads to a contradiction.

  Recall that the characteristic function of the convolution of any two distributions is the product of their characteristic functions.
  As we have that $O^{AE}\star Q_N$ converges weakly to $O^{AE}\star P$ it follows from L\'evy's continuity theorem that the characteristic functions
  $\varphi_{O^{AE}\star Q_N}=\varphi_{E}\varphi_{Q_N}$ must converge pointwise to the characteristic function $\varphi_{O^{AE}\star P} = \varphi_{E}\varphi_{P}$. As we have here that the characteristic function $\varphi_{E}$ has no roots it follows that $\varphi_{Q_N}$ must converge pointwise to $\varphi_{P}$. From L\'evy's continuity theorem it follows now that $Q_N$ converges weakly to $P$. This contradicts the fact that $\lim_{N\to\infty}\LP{Q_N}{P}>0$ as the L\'evy-Prokhorov metric metrizes the topology of weak convergence.
\end{proof}

\begin{lemma}
  \label{lemma:LP}
  We have
  \[
    \sup \set{\E{Q}{\ell(z, \xi)}\!}{\!Q\in \mc P, ~\LP{P}{Q} \leq \delta'_N}\!\leq\! \E{P}{\ell(z, \xi)} +2 \delta_N\left( \max_{z\in Z, \xi\in \Xi}\abs{\ell(z, \xi)} + \sup_{z\in Z, \xi_1\neq\xi_2\in \Xi}\frac{\abs{\ell(z, \xi_1)-\ell(z, \xi_2)}}{\norm{\xi_1-\xi_2}} \right).
  \]
\end{lemma}
\begin{proof}

  For a function $f:\Xi\to\Re$ let us define its norm as $\norm{f}_{BL}=\norm{f}_\infty+\norm{f}_L$ where $\norm{f}_\infty=\sup_{\xi\in \Xi}\abs{f(\xi)}$ and $\norm{f}_L=\sup_{\xi_1\neq\xi_2\in \Xi}\tfrac{\abs{f(\xi_1)-f(\xi_2)}}{\norm{\xi_1-\xi_2}}$.
  We associate with this norm the integral probability metric $BL(P, Q) = \sup \set{\int f(\xi)\d Q(\xi)-\int f(\xi)\d P(\xi)  }{\norm{f}_{BL}\leq 1}$.
  \citet[Corollary 2]{dudley1968distances} shows that $BL(P, Q)\leq 2 \LP{P}{Q}$.
  The remainder of the result follows from
  \begin{align*}
    & ~\sup \set{\E{Q}{\ell(z, \xi)}}{Q\in \mc P, ~2\LP{P}{Q} \leq 2\delta'_N} \\
    \leq & ~\sup \set{\E{Q}{\ell(z, \xi)}}{Q\in \mc P, ~BL(P,Q) \leq 2\delta'_N} \\
    = & ~\E{P}{\ell(z, \xi)} + \sup\set{\E{Q}{\ell(z, \xi)}-\E{P}{\ell(z, \xi)}}{Q\in \mc P, ~BL(P,Q) \leq 2\delta'_N} \\
    \leq & ~\E{Q}{\ell(z, \xi)} + \sup\set{\abs{\E{Q}{\ell(z, \xi)}-\E{P}{\ell(z, \xi)}}}{Q\in \mc P, ~BL(P,Q) \leq 2\delta'_N}\\
    = & ~\E{Q}{\ell(z, \xi)} + \norm{\ell(z, \cdot)}_{BL} \sup\set{\abs{\E{Q}{\ell(z, \xi)/\norm{\ell(z, \cdot)}_{BL}}-\E{P}{\ell(z, \xi)/\norm{\ell(z, \cdot)}_{BL}}}}{Q\in \mc P, ~BL(P,Q) \leq 2\delta'_N}\\
    \leq & ~\E{Q}{\ell(z, \xi)} + \norm{\ell(z, \cdot)}_{BL} 2\delta'_N.
  \end{align*}
  The ultimate equality form the definition of $BL(P, Q)$ with the observation that $\norm{\ell(z, \xi)/\norm{\ell(z, \cdot)}_{BL}}_{BL}\leq 1$ for all $z\in Z$.
\end{proof}

\begin{proof}[Proof of Theorem \ref{thm:consistency}]
  Assume that the event $\exists Q' ~\st~ \LP{P_N'}{Q'}\leq \delta_N,~\KL(Q', P')\leq r_N$ occurs.
  We have that $\LP{Q'}{P'}\leq TV(Q', P')\leq \sqrt{\KL(Q', P')/2}\leq \sqrt{r_N/2}$ due to \cite[Equation 4.13]{huber1981robust} and Pinkser's inequality where $TV(Q', P')$ denotes the total variation distance between $Q'$ and $P'$.
  From the triangle inequality it follows that we have here $\LP{P'_N}{O^{AE}\star P}\leq \delta_N+\sqrt{r_N/2}$. Hence, we have that
  \begin{align*}
    & \set{Q\in \mc P}{\exists Q'\in \mc P(\Xi')~\st~\LP{P_N'}{Q'}\leq \delta_N,~\KL(Q', O^{AE}\star Q)\leq r_N} \\
    \subseteq &\, \tset{Q\in \mc P}{\LP{O^{AE}\star Q}{O^{AE}\star P}\leq 2(\delta_N+\sqrt{\tfrac{r_N}{2}})}.
  \end{align*}
  Indeed, we have using the same argument that $\LP{P_N'}{Q'}\leq \delta_N,~\KL(Q', O^{AE}\star Q)\leq r_N$ implies  $\LP{P'_N}{O^{AE}\star Q}\leq \delta_N+\sqrt{r_N/2}$. Using the triangular inequality it follows that $\LP{O^{AE}\star Q}{P'_N}=\LP{P'_N}{O^{AE}\star Q}\leq \delta_N+\sqrt{r_N/2}$  and  $\LP{P'_N}{O^{AE}\star P}\leq \delta_N+\sqrt{r_N/2}$ implies $\LP{O^{AE}\star Q}{O^{AE}\star P}\leq 2(\delta_N+\sqrt{\tfrac{r_N}{2}})$.
  Hence, we have for all $z\in Z$ that
  \begin{align*}
    c(z, P) \leq \tilde c_N(z, P'_N) \defn & \sup \set{\E{Q}{\ell(z, \xi)}}{Q\in \mc P, ~I^{\delta_N}(P'_N, Q)\leq r_N}\\
    \leq & \sup \set{\E{Q}{\ell(z, \xi)}}{Q\in \mc P, ~\LP{O^{AE}\star Q}{O^{AE}\star P}\leq 2(\delta_N+\sqrt{\tfrac{r_N}{2}})}.
  \end{align*}
  Remark that the inequality $c(z, P) \leq \tilde c_N(z, P'_N)$ follows from the fact that we assume here that $\exists Q' ~\st~ \LP{P_N'}{Q'}\leq \delta_N,~\KL(Q', P')=I(Q, O^{AE}\star P)\leq r_N$ has occurred.
  As we have that $2(\delta_N+\sqrt{\tfrac{r_N}{2}})$ tends to zero, Lemma \ref{lemma:weak-convergence} guarantees that
  \begin{align*}
    \tilde c_N(z, P'_N) \leq & \, \sup \set{\E{Q}{\ell(z, \xi)}}{Q\in \mc P, ~\LP{P}{Q} \leq \delta''_N}\\
    \leq &\, \E{P}{\ell(z, \xi)} + 2\delta''_N(L+\mc L)
  \end{align*}
  for some deterministic decreasing sequence $\{\delta''_N\}_{N\geq 1}$ with $\lim_{N\to\infty}\delta''_N=0$. The second inequality is a direct consequence of Lemma \ref{lemma:LP}.
  Hence, finally we get
  \[
    \E{P}{\ell(\tilde z_N(P'_N), \xi)}=:c(\tilde z_N(P_N'), P)\leq \tilde c_N(\tilde z_N(P_N'), P'_N) \leq \inf_{z\in Z}\tilde c_N(z, P'_N) +\epsilon\leq \inf_{z\in Z} \E{P}{\ell(z, \xi)} + 2 \delta''_N (L+\mc L)+\epsilon.
  \]
  The claim follows by observing that for all $N$ sufficiently large $2 \delta''_N (L+\mc L)\leq \epsilon$ and the event $\exists Q' ~\st~ \LP{P_N'}{Q'}\leq \delta_N,~\KL(Q', P')\leq r_N$ occurs with probability tending to one following Proposition~\ref{prop:confidence}.
\end{proof}

}

\subsection{Proof of Lemma \ref{lemma:strassen-representation-finite}}
\label{sec:proof-lemma-ref}

\begin{proof}
  We first prove that $\tilde c^\delta(z, P'_N) \leq \tilde c_f^\delta(z, P'_N)$. Consider any feasible solution $(Q, \hat P', T)$ in Problem (\ref{eq:strassen-representation}). Consider
  \(
  q_\xi = Q(\xi)
  \)
  for all $\xi\in \Xi$, $\hat p'_k = \hat P'(\Xi'_k)$ and $t_{k,i}=T(\Xi'_k, \xi'_i)$ for all $1\leq i\leq \abs{\Xi'_N}$ and $0\leq k\leq K-1$.
  We have that
  \[
    \textstyle\sum_{k=0}^{K-1}\hat p'_k \log \left(\tfrac{\hat p'_k}{\sum_{\xi\in \Xi} O_\xi(\Xi'_k)\cdot q_\xi}\right) = \sum_{k=0}^{K-1}\hat P'(\Xi'_k) \log \left(\tfrac{\hat P'(\Xi'_k)}{\sum_{\xi\in \Xi} O_\xi(\Xi'_k)\cdot Q(\xi)}\right) \leq I(\hat P', Q)\leq r
  \]
  where the penultimate inequality follows from \cite{csiszar2006simple}. We have also
  \[
    \textstyle\sum_{i=1}^{\abs{\Xi'_N}} t_{k,i}= \sum_{i=1}^{\abs{\Xi'_N}} T(\Xi'_k, \xi'_i)= (\Pi_{\Xi'}T)(\Xi'_K) = \hat P'(\Xi'_k) = \hat p'_{k}
  \]
  for any $0\leq k\leq K-1$ as well as
  \[
    \textstyle\sum_{k=0}^{K-1} t_{k,i}= \sum_{k=0}^{K-1} T(\Xi'_k, \xi'_i) = (\Pi_{\Xi'_N}T)(\xi'_i)  =  P_N'(\xi'_i)
  \]
  for any $1\leq i\leq \abs{\Xi'_N}$. Finally, we remark that
  \[
    \sum_{i=1}^{\abs{\Xi'_N}} \sum_{k=0}^{K-1} b(k, i) t_{k,i} = \sum_{i=1}^{\abs{\Xi'_N}} \sum_{k=0}^{K-1} b(k, i) T(\Xi'_k, \xi'_i) = \sum_{i=1}^{\abs{\Xi'_N}} \int_{\Xi'} \one{\norm{\xi-\xi_i}\leq \delta} T(\xi', \xi'_i)  \geq 1-\delta.
  \]
  Hence, the constructed point $(q, \hat p', t)$ is feasible in Problem (\ref{eq:strassen-representation-finite}) and attains the same objective value as the point $(Q, \hat P', T)$ in Problem (\ref{eq:strassen-representation}).

  We now prove that $\tilde c^\delta(z, P'_N) \geq \tilde c_f^\delta(z, P'_N)$. Consider any feasible point $(q, \hat p', t)$ in Problem (\ref{eq:strassen-representation-finite}) and let $Q(\xi)=q_\xi$ for all $\xi\in \Xi$. Define first the measures $(O\star Q)_{\Xi'_k}\in \mc P(\Xi'_k)$ as the restriction of the measure $O\star Q$ to the set  $\Xi'_k\subseteq \Xi'$ for all $0 \leq k \leq K-1$. That is, these measures are defined through
  \(
  (O\star Q)_{\Xi'_k}(B) = \tfrac{(O\star Q)(B\cap\Xi'_k)}{(O\star Q)(\Xi'_k)}
  \)
  for all measurable subsets $B$ of $\Xi'$ and $0 \leq k \leq K-1$. Let now \( \hat P' = \sum_{k=0}^{K-1} \hat p'_k \cdot (O\star Q)_{\Xi'_k} \) and \( T = \sum_{i=1}^{\abs{\Xi'_N}} \sum_{k=0}^{K-1} t_{k,i} \cdot (\delta_{\xi'_i}\otimes (O\star Q)_{\Xi'_k})\).
  First, we remark that we have
  \begin{align*}
    (\Pi_{\Xi'}T)(B) =& \textstyle \sum_{j=1}^{\abs{\Xi'_N}} \sum_{i=1}^{\abs{\Xi'_N}} \sum_{k=0}^{K-1} t_{k,i} \cdot (\delta_{\xi'_i}(\xi'_j)\cdot (O\star Q)_{\Xi'_k}(B)) \\
    = & \textstyle \sum_{i=1}^{\abs{\Xi'_N}} \sum_{k=0}^{K-1} t_{k,i} \cdot (O\star Q)_{\Xi'_k}(B) \\
    = & \textstyle \sum_{k=0}^{K-1} \sum_{i=1}^{\abs{\Xi'_N}} t_{k, i} \cdot (O\star Q)_{\Xi'_k}(B) \\
    = & \textstyle \sum_{k=0}^{K-1} \hat p'_{k} \cdot (O\star Q)_{\Xi'_k}(B) = \hat P'(B)
  \end{align*}
  for any measurable subset $B\subseteq \Xi'_k$ for all $0 \leq k \leq K-1$. Hence, $\Pi_{\Xi'}T=\hat P'$. Likewise, we have that
  \begin{align*}
    (\Pi_{\Xi'_N}T)(\xi') =& \textstyle \sum_{l=0}^{K-1} \sum_{i=1}^{\abs{\Xi'_N}} \sum_{k=0}^{K-1} t_{k,i} \cdot (\delta_{\xi'_i}(\xi')\cdot (O\star Q)_{\Xi'_k}(\Xi'_l)) \\
    = & \textstyle \sum_{i=1}^{\abs{\Xi'_N}} \sum_{k=0}^{K-1} t_{k,i} \cdot \delta_{\xi'_i}(\xi') \\
    = & \textstyle \sum_{i=1}^{\abs{\Xi'_N}} P_N'(\xi'_i) \cdot \delta_{\xi'_i}(\xi') = P_N'(\xi')
  \end{align*}
  for all $\xi'\in \Xi'_N$. Hence, $\Pi_{\Xi'_N}T=P'_N$. Furthermore, as by construction the Radon-Nikodym derivative $(\d \hat P'/\d (O\star Q))(\xi) = \hat p'_k/(O\star Q)(\Xi'_k)$ for all $\xi\in \Xi_k$, $0 \leq k \leq K-1$ we have
  \[
    I(\hat P', Q) = \textstyle\sum_{k=0}^{K-1} \int_{\Xi'_k} \hat P'(\xi') \log\left(\frac{\d \hat P'}{\d (O\star Q)}(\xi')\right) = \sum_{k=0}^{K-1} \hat p'_k \log\left( \hat p'_k/(O\star Q)(\Xi'_k) \right) \leq r.
  \]
  Finally, we have
  \begingroup
  \allowdisplaybreaks
  \begin{align*}
   \textstyle\int_{\Xi'\times\Xi'_N}\one{\norm{\xi'-\xi'_i}\leq \delta}\d T(\xi', \xi_i') = & \textstyle\sum_{l=0}^{K-1}\int_{\Xi'_l\times \Xi'_N}\one{\norm{\xi'-\xi'_j}\leq \delta}\d T(\xi', \xi_j')\\
    = & \textstyle\sum_{l=0}^{K-1}\int_{\Xi'_l\times \Xi'_N}\one{\norm{\xi'-\xi'_j}\leq \delta}\d \sum_{i=1}^{\abs{\Xi'_N}} \sum_{k=0}^K t_{k,i} \cdot (\delta_{\xi'_i}\otimes (O\star Q)_{\Xi'_k})(\xi', \xi'_j)\\
    = & \textstyle\sum_{l=0}^{K-1}\sum_{i=1}^{\abs{\Xi'_N}} \sum_{k=0}^K t_{k,i} \int_{\Xi'_l\times \Xi'_N}\one{\norm{\xi'-\xi'_j}\leq \delta}\d(\delta_{\xi'_i}\otimes (O\star Q)_{\Xi'_k})(\xi', \xi'_j)\\
    = & \textstyle\sum_{i=1}^{\abs{\Xi'_N}} \sum_{k=0}^K t_{k,i} \int_{\Xi'_k\times \Xi'_N}\one{\norm{\xi'-\xi'_j}\leq \delta}\d(\delta_{\xi'_i}\otimes (O\star Q)_{\Xi'_k})(\xi', \xi'_j)\\
    = & \textstyle\sum_{i=1}^{\abs{\Xi'_N}} \sum_{k=0}^K t_{k,i} \int_{\Xi'_k}\one{\norm{\xi'-\xi'_i}\leq \delta}\d((O\star Q)_{\Xi'_k})(\xi')\\
    = & \textstyle\sum_{i=1}^{\abs{\Xi'_N}} \sum_{k=0}^K t_{k,i} b(k, i)\geq 1-\delta.
  \end{align*}
  \endgroup
  Hence, the constructed point $(Q, \hat P', T)$ is feasible in Problem (\ref{eq:strassen-representation}) and attains the same objective value as the point $(q, \hat p', t)$ in Problem (\ref{eq:strassen-representation-finite}).
\end{proof}

\subsection{Proof of Theorem \ref{thm:dual-formulations}}
\label{sec:proof-theor-dual}

\begin{proof}
  Introduce the Lagrangian function of the optimization problem (\ref{eq:strassen-representation}) as
  \begin{align*}
    & L(Q, \hat P', T ~; ~\beta, u, v, \gamma) \\[0.3em]
    = & \int_\Xi\ell(z, \xi) \, \d Q(\xi) + \beta(r-I(\hat P', Q)) - \int_{\Xi'\times \Xi'_N} u(\xi') \, \d T(\xi', \xi'_i) + \int_{\Xi'} u(\xi') \, \d \hat P'(\xi')\\[0.3em]
    & \quad + \int_{\Xi'\times \Xi'_N} v(\xi'_i)\, \d T(\xi', \xi'_i) - \int_{\Xi'_N} v(\xi'_i) \, \d P_N(\xi'_i) +  \int_{\Xi'\times\Xi'_N}\gamma \one{\norm{\xi'-\xi'_i}\leq \delta}\d T(\xi', \xi_i') - \gamma(1-\delta)\\[0.3em]
    = & \beta r - \int_{\Xi'_N} v(\xi'_i) \, \d P_N(\xi'_i) - \gamma(1-\delta)  + \int_{\Xi'\times\Xi'_N} \gamma \one{\norm{\xi'-\xi'_i}\leq \delta} + v(\xi'_i) - u(\xi')\, \d T(\xi', \xi'_i) \\[0.3em]
    & \quad + \int_\Xi\ell(z, \xi) \, \d Q(\xi) + \int_{\Xi'} u(\xi') \, \d \hat P'(\xi') - \beta I(\hat P', Q).
  \end{align*}
  We have $\tilde c^\delta(z, P'_N) =\max_{Q\in \mc P(\Xi), \hat P'\in \mc M_+(\Xi'), T\in \mc M_+(\Xi'\times \Xi_N)} \min_{\beta\geq 0, u:\Xi'\to\Re, v:\Xi_N\to\Re, \gamma\geq 0} L(Q, \hat P', T ~; ~\beta, u, v, \gamma)$.
  The associated dual function is
  \begin{align*}
    & g(\beta, u, v, \gamma)\\[0.3em]
     = & \max \set{L(Q, \hat P', T ~; ~\beta, u, v, \gamma)}{{Q\in \mc P(\Xi), ~\hat P' \in \mc M_+(\Xi'), ~T\in \mc M_+(\Xi'\times\Xi_N')}} \\
                           = &  \beta r - \int_{\Xi'_N} v(\xi'_i) \, \d P_N(\xi'_i) - \gamma(1-\delta) + \chi_\infty\{\gamma \one{\norm{\xi'-\xi'_i}\leq \delta} +  v(\xi'_i) \leq u(\xi')~~ \forall \xi'\in \Xi', \,\xi'_i\in \Xi'_N \}+\\[0.3em]
    & \max \set{ \int_\Xi\ell(z, \xi) \, \d Q(\xi) + \int_{\Xi'} u(\xi') \, \d \hat P'(\xi') - \beta \int_{\Xi'} \log\left(\frac{\d \hat P'}{\d Q}(\xi')\right) \d \hat P'(\xi') }{Q\in \mc P(\Xi), ~\hat P' \in \mc M_+(\Xi')}\\[0.3em]
    = &  \beta r - \int_{\Xi'_N} v(\xi'_i) \, \d P_N(\xi'_i) - \gamma(1-\delta) + \chi_\infty\{\gamma \one{\norm{\xi'-\xi'_i}  \leq \delta}+ v(\xi'_i)\leq u(\xi')~~ \forall \xi'\in \Xi', \,\xi'_i\in \Xi'_N \}+\\[0.3em]
    & \max_{\xi\in\Xi} \left[\ell(x, \xi) + \beta \int_{\Xi'} \exp\left(\frac{u(\xi')}{\beta}+1\right) \d O_\xi(\xi')\right]
  \end{align*}
  where we establish the ultimate equality by observing that
  \begingroup
  \allowdisplaybreaks
  \begin{align*}
    & \max \set{ \int_\Xi\ell(z, \xi) \, \d Q(\xi) + \int_{\Xi'} u(\xi') \, \d \hat P'(\xi') - \beta \int_{\Xi'} \log\left(\frac{\d \hat P'}{\d (O\star  Q)}(\xi')\right) \d \hat P'(\xi') }{Q\in \mc P(\Xi), ~\hat P' \in \mc M_+(\Xi')}\\[0.3em]
    = & \max \set{ \int_\Xi\ell(z, \xi) \, \d Q(\xi) + \int_{\Xi'} u(\xi') R(\xi') - \beta \log\left(R (\xi')\right)  R(\xi') \, \d (O\star Q)(\xi') }{Q\in \mc P(\Xi), ~R(\xi')\geq 0}\\[0.3em]
    = & \max \set{ \int_\Xi\ell(z, \xi) \, \d Q(\xi) + \int_{\Xi'} \left[\max_{R\geq 0} u(\xi') R - \beta\log\left(R\right)  R\right] \, \d (O\star Q)(\xi') }{Q\in \mc P(\Xi)}\\[0.3em]
    = & \max \set{ \int_\Xi\ell(z, \xi) \, \d Q(\xi) + \int_{\Xi'} \beta \exp\left(\frac{u(\xi')}{\beta} - 1\right) \, \d (O\star Q)(\xi') }{Q\in \mc P(\Xi)} \\[0.5em]
    = & \max_{\xi\in \Xi} \,\ell(z, \xi) + \beta \int_{\Xi'} \exp\left(\frac{u(\xi')}{\beta} - 1\right) \, \d O_\xi(\xi')
  \end{align*}
  \endgroup
  where the first equality follows from the substitution $R=\tfrac{\d \hat P'}{\d (O\star Q)}$. The fourth equality follows from the identity $\max_{R\geq 0} a R - \beta\log\left(R\right)R = a R^\star(a, \beta) - \beta\log\left(R^\star(a, \beta)\right)R^\star(a, \beta) = \beta\exp(\tfrac{a}{\beta}-1)$  with $R^\star(a, \beta)=\exp\left(\tfrac{a}{\beta}-1\right)$ for all $a$.
  Weak duality implies that $\tilde c^\delta(z, P'_N)\leq \min \set{g(\beta, u, v, \gamma)}{\beta\geq 0,~u:\Xi'\to\Re,~v:\Xi_N\to\Re, \gamma\geq 0}=~(\ref{eq:dual-formulation})$. Strong duality, i.e., $c^\delta(z, P'_N)=(\ref{eq:dual-formulation})$, follows from the assumed existence of the Slater points $Q_s$ and $\hat P'_s$ \citep[Proposition 5.3.1]{bertsekas2009convex} in case $\Xi'$ is finite.

  In case $\Xi'$ is not finite we consider the slightly weaker dual
  \begin{equation}
  \label{eq:dual-formulation-weak}
  \begin{array}{r@{~~}l}

    {\displaystyle\inf} & \displaystyle \beta r\! -\!\int_{\Xi'} v(\xi') \, \d P'_N(\xi')  + \gamma (\delta-1) + \max_{\xi\in\Xi}\! \left[\ell(z, \xi) + \beta \!\!\int_{\Xi'}\!\!\! \exp\left(\frac{u(\xi')}{\beta}-1\right)\d O_\xi(\xi')\right]\\[0.3em]
    \st & \beta\geq 0, ~u\in \mc {PC}(\Xi'), ~v:\Xi'_N\to\Re,~\gamma\geq 0,\\[0.3em]
                        & \gamma \one{\norm{\xi'-\xi'_i}  \leq \delta}+ v(\xi'_i)\leq u(\xi')  \quad \forall \xi'\in \Xi', \,\xi'_i\in \Xi'_N.
  \end{array}
\end{equation}
where $u(\xi)\in \mc {PC}(\Xi')$ indicates that the dual variable $u$ is now restricted to be piecewise constant on the pieces $\Xi'_k$ for $0\leq k\leq K-1$ as defined in Section \ref{sec:strass-repr}. Taking $u(\xi')=u_k$ for all $\xi'\in \Xi'_k, ~0\leq k\leq K-1$ and $v(\xi_i)=v_i$ for $1\leq i\leq \abs{\Xi'_N}$ the weak dual \eqref{eq:dual-formulation-weak} is clearly equivalent to
\begin{equation}
  \begin{array}{r@{~~}l}
    {\displaystyle\inf} &  \beta r -\sum_{i=1}^{\abs{\Xi'_N}} v_i P'_N(\xi'_i)  + \gamma (\delta-1) + \max_{\xi\in\Xi}\left[\ell(z, \xi) + \beta \sum_{k=0}^{K-1} \exp\left(\frac{u_k}{\beta}-1\right)O_\xi(\Xi'_k)\right]\\[0.3em]
    \st & \beta\geq 0, ~u\in \Re^K, ~v\in \Re^{\abs{\Xi'_N}},~\gamma\geq 0,\\[0.3em]
                        & \gamma b(k, i) + v_i\leq u_k  \quad  0\leq k\leq K-1, ~1\leq i\leq \abs{\Xi'_N}.
  \end{array} 
\end{equation}
as $\one{\norm{\xi'-\xi'_i}  \leq \delta} = b(k, i)$ for all $\xi'\in \Xi'_k$ and $\xi'_i\in \Xi'_N$. The remainder of the theorem follows from Lemma \ref{sec:proof-theor-refthm:d} and Lemma \ref{lemma:strassen-representation-finite}.
\end{proof}

\begin{lemma}
  \label{sec:proof-theor-refthm:d}
  Let $\Xi'_N={\rm{supp}}(P'_N)$ denote the support of $P'_N$ and let $\Xi$ be finite.
  Consider the finite convex minimization problem
  \begin{equation}
    \label{eq:dual-formulation-special-finite}
    \begin{array}{r@{~~}l}
      {\displaystyle\inf} &  \beta r -\sum_{i=1}^{\abs{\Xi'_N}} v_i P'_N(\xi'_i)  + \gamma (\delta-1) + \max_{\xi\in\Xi}\left[\ell(z, \xi) + \beta \sum_{k=0}^{K-1} \exp\left(\frac{u_k}{\beta}-1\right)O_\xi(\Xi'_k)\right]\\[0.3em]
      \st & \beta\geq 0, ~u\in \Re^K, ~v\in \Re^{\abs{\Xi'_N}},~\gamma\geq 0,\\[0.3em]
                          & \gamma b(k, i) + v_i\leq u_k  \quad  0\leq k\leq K-1, ~1\leq i\leq \abs{\Xi'_N}.
    \end{array} 
  \end{equation}
  If there exists $Q_s\in \mc P(\Xi)$, $\hat P'_s\in \mc P(\Xi')$ with $I(\hat P'_s, Q_s)<r$ and $\LP{\hat P'_s}{P'_N}<\delta$ then $\eqref{eq:strassen-representation-finite}=\eqref{eq:dual-formulation-special-finite}$.
\end{lemma}
\begin{proof}
  It can be shown that problem \eqref{eq:dual-formulation-special-finite} is the Lagrangian dual of Problem \eqref{eq:strassen-representation-finite}. As this Lagrangian dual can be derived following the proof of Theorem \ref{thm:dual-formulations} almost verbatim, we omit it here for the sake of brevity. Furthermore, the construction used in the proof of Lemma \ref{lemma:strassen-representation-finite} shows that the existence of $Q_s\in \mc P(\Xi)$, $\hat P'_s\in \mc P(\Xi')$ with $I(\hat P'_s, Q_s)<r$ and $\LP{\hat P'_s}{P'_N}<\delta$ ensures the existence of a Slater point in the finite dimensional optimization Problem \eqref{eq:strassen-representation-finite}.  Strong duality, i.e., $\eqref{eq:strassen-representation-finite}=\eqref{eq:dual-formulation-special-finite}$, follows then from \citep[Proposition 5.3.1]{bertsekas2009convex}.
\end{proof}

\end{document}